\documentclass[a4paper,11pt]{article}
\usepackage{amsfonts}
\usepackage{amssymb}
\usepackage{amssymb,amsfonts,amsmath,amsthm,cite,color}
\usepackage{dsfont}
\usepackage{epsfig}
\usepackage{mathrsfs}
\usepackage{algorithmic}
\usepackage{graphicx}
\usepackage{epstopdf}
\usepackage{setspace}
\usepackage{caption}
\usepackage{graphicx, subfig}
\usepackage{appendix}
\usepackage{color}
\usepackage{bm}
\renewcommand{\thefootnote}{}

\newtheorem{thm}{Theorem}[section]

\newtheorem{lem}[thm]{Lemma}

\newtheorem{defn}[thm]{Definition}

\newtheorem{Claim}{Claim}

\makeatletter \@addtoreset{equation}{section}

\begin{document}
\begin{center}
\begin{spacing}{1.5}
{\Large \bf Packing internally disjoint Steiner paths of data center networks}
\end{spacing}
\end{center}

\begin{center}
{Wen-Han Zhu}$^{a}$,
{Rong-Xia Hao}$^{a,}$
\renewcommand{\thefootnote}{\fnsymbol{footnote}}\footnote{Corresponding author},
{Jou-Ming Chang}$^{b}$,
{Jaeun Lee}$^{c}$

$^{a}$School of Mathematics and Statistics, Beijing Jiaotong University,\\ Beijing 100044, P.R. China\\[10pt]
$^{b}$Institute of Information and Decision Sciences, National Taipei University of Business, Taipei 10051, Taiwan\\[10pt]
$^{c}$Department of Mathematics, Yeungnam University, 280 Daehak-ro, Gyeongsan, Gyeongbuk 38541, Korea
\end{center}
\footnote{E-mail addresses: rxhao@bjtu.edu.cn (R.-X. Hao)}


\vskip 3mm \noindent {\bf Abstract:}
Let $S\subseteq V(G)$ and $\pi_{G}(S)$ denote the maximum number $t$ of edge-disjoint
paths $P_{1},P_{2},\ldots,P_{t}$ in a graph $G$ such that
$V(P_{i})\cap V(P_{j})=S$ for any $i,j\in\{1,2,\ldots,t\}$ and $i\neq j$.
If $S=V(G)$, then $\pi_{G}(S)$ is the maximum number of edge-disjoint spanning paths in $G$.
It is proved [Graphs Combin., 37 (2021) 2521-2533] that deciding whether $\pi_G(S)\geq r$ is NP-complete for a given $S\subseteq V(G)$.
For an integer $r$ with $2\leq r\leq n$, the $r$-path connectivity of a graph $G$
is defined as $\pi_{r}(G)=$min$\{\pi_{G}(S)|S\subseteq V(G)$ and $|S|=r\}$, which is a generalization of tree connectivity.
  In this paper, we study the $3$-path connectivity of the $k$-dimensional data center network with $n$-port switches $D_{k,n}$ which has significate role in  the cloud computing,
and prove that $\pi_{3}(D_{k,n})=\lfloor\frac{2n+3k}{4}\rfloor$ with $k\geq 1$ and $n\geq 6$.


\section{Introduction}
Connectivity, one of the most important concepts in graph theory,
clearly describes the propagation mode between vertices.
For a connected graph $G$ with vertex set $V(G)$ and edge set $E(G)$,
the {\it connectivity} $\kappa(G)$ of $G$ was proposed by Whitney \cite{Whitney},
which is defined as $\kappa(G)=$min$\{\kappa_{G}(u,v)|\{u,v\}\subseteq V(G)\}$,
where $\kappa_{G}(u,v)$ is the maximum number of internally disjoint paths connecting $(u,v)$ in $G$.
Dirac \cite{Dirac} proved that there exists a path containing any $r$ vertices in $G$,
where $G$ is $(r-1)$-connected.
As a generalization,
Hager \cite{Hager} revised this problem as that how many internally disjoint paths contain
any $r$ vertices.
Let $S$ be a vertex subset of $G$ with $|S|\geq 2$.
A path $P$ in $G$ is called an {\it Steiner path} or {\it $S$-path}, if $S\subseteq V(P)$.
The $S$-paths $P_{1},P_{2},\ldots,P_{t}$ are said to be {\it internally disjoint $S$-paths},
simply say ID$S$-paths,
if $V(P_{i})\cap V(P_{j})=S$ and $E(P_{i})\cap E(P_{j})=\emptyset$ for any integer $1\leq i\neq j\leq t$.
For an integer $r$ with $2\leq r\leq |V(G)|$,
the {\it $r$-path connectivity}, denoted by $\pi_{r}(G)$,
is defined as $\pi_{r}(G)=$min$\{\pi_{G}(S)|S\subseteq V(G)$ and $|S|=r\}$,
where $\pi_{G}(S)$ denote the maximum number of internally disjoint $S$-paths.
Clearly, $\pi_{1}(G)=\delta(G)$ and $\pi_{2}(G)=\kappa(G)$.
Similarly as $r$-path connectivity,
Hager \cite{Hager1985} and Chartrand et al. \cite{Chartrand} defined
{\it the $r$-tree connectivity} $\kappa_{r}(G)$ by replacing path with tree.
\\
\indent
In recent years, with the widespread application of cloud computing technology,
the development of web search, online gaming, email, cloud storage and infrastructure
services such that data centers can host more and more servers.
To support larger data center networks, researchers have proposed a variety of new network structures.
{\it Switch-centric networks} and {\it server-centric networks} were proposed based on
computational intensive tasks such as routing which are put into the switches or on the servers.
The server-centric data center networks can greatly reduce the cost of network hardware.
Recently, Guo et al.\cite{Guo} proposed a server-centric data center network called DCell,
which can be used to handle a large number of servers.
It initiated an alternative design called server-centric DCNs and inspired a number of novel DCN designs,
such as BCube \cite{GuoLu}, FiConn \cite{Li}, CamCube \cite{Libdeh}, and so on.
DCell has the advantages of exponential scalability, small diameter,
large binary width, high network capacity, and high fault tolerance.
Some basic properties and algorithms of DCell networks have recently been studied,
such as symmetry
and edge symmetry \cite{Gu,Kliegl}, diameter \cite{Kliegl},
connectivity and restricted connectivity \cite{Guo,Manzano,Wang},
broadcasting and fault-tolerant routing \cite{Guo},
vertex-pancyclicity \cite{Hao},
Hamiltonicity and fault-tolerant Hamiltonicity \cite{Qin,X.Wang},
one-to-one disjoint path cover \cite{WangAndFan},
diagnosability analysis \cite{Gu}
and the completely independent spanning trees \cite{Chen,Pai} .
These measurement results show that the $D_{k,n}$ has good communication performance.
\\
\indent
The reader can refer \cite{Ma,LS,Qin2021,Zhao2019Hao} et al. for many results about $\kappa_{r}(G)$,
but there are a few of results about $\pi_{r}(G)$.
For example, $\pi_r{(K_n)}$ and $\pi_r{(K_{s,t})}$ are studied in \cite{Hager}.
The bounds of $3$-path connectivity for the Lexicographic product of graphs
have been characterized in \cite{Mao}.
Moreover, Li et al.\cite{Qin2021} had proved that
deciding whether $\pi_G(S)\geq r$ with $S\subseteq V(G)$ is NP-complete for any fixed integer $r\geq 1$.
The $3$-path connectivity of $k$-ary $n$-cubes and hypercubes had been characterized
in \cite{Zhu2022} and \cite{Zhu2022Hao}, respectively.
\\
\indent
In this paper, we solve the problem of how to pack the maximal number of internally disjoint paths on
$k$-dimensional data center network with $n$-port switches $D_{k,n}$
and prove that $\pi_{3}(D_{k,n})=\lfloor\frac{2n+3k}{4}\rfloor$ with $k\geq 1$ and $n\geq 6$.
In Section $2$, some terminologies
and notations used are introduced.
Moreover, some properties of $D_{k,n}$ and lemmas are given.
In Section $3$,
the results are derived.
In Section $4$, the paper is concluded.

\section{Preliminaries and preparation}
In this paper, we only consider a simple, undirected, connected graph $G=(V(G),E(G))$.
For any vertex $x\in V(G)$, the vertex set $\{y\in V(G): xy\in E(G)\}$
is the {\it neighborhood} of $x$ in $G$, denoted by $N_{G}(x)$.
In addition, $N_{G}[x]=N_{G}(x)\cup \{x\}$.
The {\it degree} of a vertex $x$ in a graph $G$, denoted by $d_{G}(x)$, is the number of edges of
$G$ incident with $x$.
As $G$ is a simple graph, $d_{G}(x)=|N_{G}(x)|$.
A graph is {\it $d$-regular} if $d_{G}(x)=d$ for any vertex $x\in V(G)$.
Let $X\subseteq V(G)$,
then $G[X]$ is referred to as the subgraph of $G$ induced by $X$
whose vertex set is $X$ and whose edge set consists all of edges of $G$
which have both ends in $X$.
Let $P$ be a nontrivial path in $G$,
then it must contain two vertices satisfying that each of them has exactly degree one,
such vertices are called the {\it terminal vertices} of $P$.
Moreover, a vertex with degree two in $P$ is called the {\it internal vertex} of $P$.
A path $P$ in $G$ with $x$ and $y$ as its two terminal vertices is said to be an {\it $(x,y)$-path},
denoted by $P[x,y]$ or $P[y,x]$.
Let $x\in V(G)$ and $Y\subseteq V(G)\backslash\{x\}$.
A path which starts at the vertex $x$, and ends at a vertex of $Y$ is an
{\it $(x,Y)$-path} whose internal vertices do not belong to $Y$.
Let $S=\{x,y,z\}$ be any vertex subset of $G$.
An {\it $S$-path} in $G$ is a path with two vertices of $S$ as terminal vertices and
the remaining vertex of $S$ as internal vertex.
For any two distinct vertices $x$ and $y$ of $V(G)$,
the {\it distance} $d_{G}(x,y)$, abbreviated as $d(x,y)$,
between $x$ and $y$ is the length of a shortest path connecting $x$ and $y$ in $G$.
Let $P_{1},P_{2},\ldots,P_{r}$ be some paths in $G$
and $\mathcal{P}=\{P_{1},P_{2},\ldots,P_{r}\}$,
denoted by $V(\mathcal{P})=\bigcup_{i=1}^{r}V(P_{i})$,
$E(\mathcal{P})=\bigcup_{i=1}^{r}E(P_{i})$
and $d_{\mathcal{P}}(u)=\sum_{i=1}^{r}d_{P_{i}}(u)$ for any $u\in V(\mathcal{P})$.
Let $A,B$ be subgraphs of $G$ and $V(A)\cap V(B)=\emptyset$.
Then let $E[A,B]$ be the set of edges between $A$ and $B$ in $G$.
For more notation and terminology, we refer to \cite{Bondy}, unless otherwise stated.
\\
\indent
Given a positive integer $m$, let $[m]=\{1,2,\ldots,m\}$ and $\langle m\rangle=\{0,1,\ldots,m-1\}$.
For any integer $k\geq 0$ and $n\geq 2$, the {\it $k$-dimensional data center network}
with $n$-port switches is denoted by $D_{k,n}$.
Let $t_{k,n}=|V(D_{k,n})|$ and $t_{i,n}=t_{i-1,n}\times (t_{i-1,n}+1)$ for each $i\in [k]$,
and $t_{0,n}=n$.
The definition of $D_{k,n}$ is as follows.
\begin{defn}\label{Guo}
The $k$-dimensional data center network with $n$-port switches $D_{k,n}$ is a graph with the vertex set
$V(D_{k,n})=\{(a_{k},a_{k-1},\ldots, a_{0}):a_{i}\in\langle t_{i-1,n}+1\rangle, i\in [k]$
and $a_{0}\in \langle n\rangle\}$.
The network $D_{k,n}$ can be defined recursively as follows.
Let $D_{0,n}$ be the complete graph $K_{n}$.
For $k> 0$, $D_{k,n}$ is constructed from $t_{k-1,n}+1$ disjoint copies of $D_{k-1,n}$,
denoted by $D_{k,n}=D_{k-1,n}^{0}\bigoplus D_{k-1,n}^{1}\bigoplus \dots \bigoplus D_{k-1,n}^{t_{k-1,n}}$.
Let $D_{k-1,n}^{i}$ be the $i$-th copy for $i\in \langle t_{i-1,n}+1\rangle$.
Each pair of distinct copies $D_{k-1,n}^{i}$ and $D_{k-1,n}^{j}$ of $D_{k-1,n}$ is connected by the
$k$-dimensional edge, say $xy$, according to the {\bf Connection rule}.
If $xy$ is a $k$-dimensional edge, then $y$ is called the (unique) $k$-dimensional neighbour of $x$.
A vertex $a\in V(D_{k-1,n}^{r})$ is labeled a $(k+1)$-tuple $(r,a_{k-1},\ldots,a_{0})$,
where $a_{i}\in \langle t_{i-1,n}+1\rangle$ for $0< i\leq k-1$ and $a_{0}\in \langle n\rangle$.
The suffix $(a_{j},a_{j-1},\ldots,a_{0})$ has the unique uid$_{j}(a)$, where
uid$_{j}(v)=a_{0}+\sum_{\ell=1}^{j}(a_{\ell}\times t_{\ell-1,n})$.
In $D_{k,n}$, each vertex is uniquely identified by its $(k+1)$-tuple and the $(k+1)$-tuple
can also be derived from its unique uid$_{k}$.
\\
\indent {\bf Connection rule:} For each pair of $D_{k-1,n}$, say $D_{k-1,n}^{r}$ and $D_{k-1,n}^{s}$ ($r< s$),
the vertex $a=(a_{k},a_{k-1},\ldots,a_{0})$ in $D_{k-1,n}^{r}$ is incident with the vertex $b=(b_{k},b_{k-1},\ldots,b_{0})$
in $D_{k-1,n}^{s}$ if and only if $a_{k}=r=uid_{k-1}(b)$
and $b_{k}=s=uid_{k-1}(a)+1$.
\end{defn}

\indent
Several data center networks with small values of $k$ and $n$ are shown in Figure~\ref{DCell}.
\begin{figure}[ht]
\begin{center}
\scalebox{0.5}[0.5]{\includegraphics{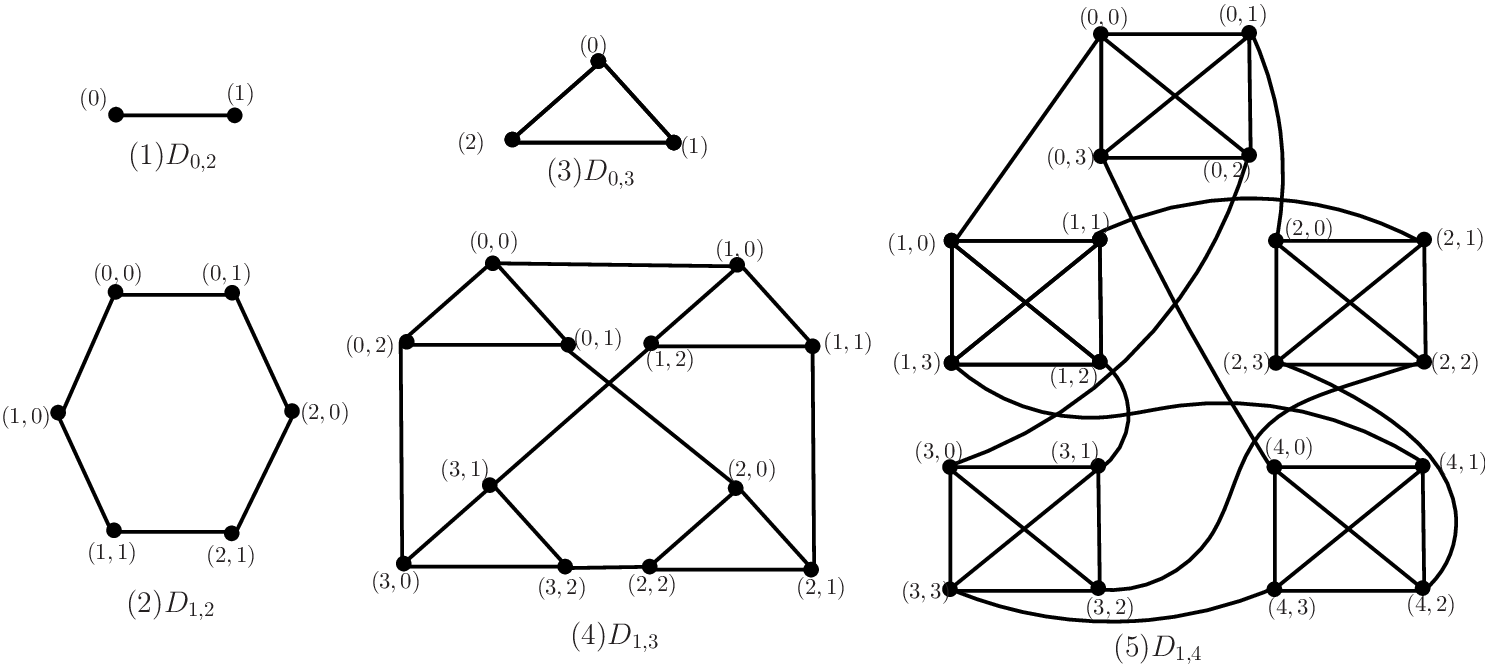}}\\
\captionsetup{font={small}}
\caption{Several DCells}\label{DCell}
\end{center}
\end{figure}
\\
\indent
According to the definition of $D_{k,n}$, Lemma \ref{lem1} is established.

\begin{lem}\label{lem1}
Let $D_{k,n}$ be the $k$-dimensional data center network with $n$-port
switches for $k\geq 0$ and $n\geq 2$.
Then the following four conditions hold.
\\
\indent
\emph{(1)} $D_{k,n}$ is $(n+k-1)$-regular and $\kappa(D_{k,n})=\lambda(D_{k,n})=n+k-1$.
\\
\indent
\emph{(2)} For $k\geq 1$, $D_{k,n}$ consists of $t_{k-1,n}+1$ copies of $D_{k-1,n}$,
denoted by $D_{k-1,n}^{r}$ for each $r\in \langle t_{k-1,n}+1\rangle$.
For each vertex $a\in D_{k-1,n}^{r}$ with $a=(r,a_{k-1},\ldots,a_{0})$,
where $a_{i}\in \langle t_{i-1,n}+1\rangle$ for $0\leq i\leq k-1$,
the vertex $a$ has only one $k$-dimensional neighbour, say $a^{\prime}$.
\\
\indent
\emph{(3)}
Let $a$ and $b$ be any two distinct vertices in $D_{k,n}$.
If $a,b\in V(D_{k-1,n}^{r})$,
then the $k$-dimensional neighbours
$a^{\prime}$ and $b^{\prime}$ of $a$ and $b$, respectively, belong to two distinct copies of $D_{k-1,n}$.
Moreover, there exists a unique $k$-dimensional edge between any two distinct copies of $D_{k-1,n}$.
\\
\indent
\emph{(4)} For any three distinct vertices $a,b$ and $c$ in $D_{k,n}$,
then $|N_{D_{k,n}}(a)\cap
N_{D_{k,n}}(b)\cap N_{D_{k,n}}(c)|\leq n-3$ with $n\geq 4$.
Moreover, $|N_{D_{k,n}}(a)\cap N_{D_{k,n}}(b)\cap N_{D_{k,n}}(c)|= n-3$
if and only if $\{a,b,c\}\subseteq V(K_{n})$, where $K_{n}$ as a subgraph of $D_{k,n}$ is a complete graph.
\end{lem}

The following Lemmas \ref{lem6}-\ref{thm21} are important for studying the path connectivity.

\begin{lem}\label{lem6}\cite{Bondy}
Let $G$ be a $k$-connected graph, $x$ and $y$ be
a pair of distinct vertices in $G$. Then there exist $k$ internally disjoint
paths $P_{1},P_{2},\ldots,P_{k}$ in $G$ connecting $x$ and $y$.
\end{lem}


\begin{lem}\label{lem4}(Fan Lemma\cite{Bondy})
Let $G=(V,E)$ be a $k$-connected graph,
$x$ be a vertex of $G$, and $Y\subseteq V\backslash\{x\}$
be a set of at least $k$ vertices of $G$.
Then there exists a $k$-fan in $G$ from $x$ to $Y$
(where, a family of $k$ internally disjoint $(x,Y)$-paths whose terminal vertices are distinct
is referred to as a $k$-fan from $x$ to $Y$).
\end{lem}

\begin{lem}\label{lem2}\cite{Zhu2022}
Let $G$ be a $k$-regular connected graph.
Then $\pi_{3}(G)\leq \lfloor\frac{3k-r}{4}\rfloor$,
where $r=$max$\{|N_{G}(x)\cap N_{G}(y)\cap N_{G}(z)|:\{x,y,z\}\subseteq V(G)\}$.
\end{lem}

\begin{lem}\label{thm1}\cite{Hager}
Let $K_{n}$ be the complete graph with order $n$ and $k$ be an integer with $n\geq k\geq 2$.
Then
$\pi_{k}(K_{n})=\lfloor \frac{2n+k^{2}-3k}{2(k-1)}\rfloor.$
\end{lem}

\begin{lem}\label{thm21}
Let $G$ be a $k$-regular connected graph with $\pi_{3}(G)\geq 2$,
$x,y$ and $z$ be any three distinct vertices of $G$ with
$|N_{G}(x)\cap N_{G}(y)\cap N_{G}(z)|=r$ and $3k-4\pi_{3}(G)\geq r+\ell$.
Then (a) and (b) hold.
\\
\indent
(a)
There exist a vertex $u$ and a set $\mathcal{T}$ of $\pi_{3}(G)$
internally disjoint $\{x,y,z\}$-paths in $G$ such that
$u\in (N_{G}(x)\cup N_{G}(y)\cup N_{G}(z))\backslash V(\mathcal{T})$ for $\ell=1$.
\\
\indent
(b)
There exist two distinct vertices $u,v$ and a set $\mathcal{T}$ of $\pi_{3}(G)$
internally disjoint $\{x,y,z\}$-paths in $G$ such that
$u,v\in (N_{G}(x)\cup N_{G}(y)\cup N_{G}(z))\backslash V(\mathcal{T})$ for $\ell=3$.
\end{lem}

The proof of the Lemma \ref{thm21} is shown in Appendix.


\section{The $3$-path-connectivity of data center networks}

\begin{lem}\label{lem5}
Let $D_{k,n}=D_{k-1,n}^{0}\bigoplus D_{k-1,n}^{1}\bigoplus \dots \bigoplus D_{k-1,n}^{t_{k-1,n}}$
be the $k$-dimensional data center network with $n$-port switches for $n\geq 3$,
and $H=D_{k,n}[\bigcup_{q=i_{1}}^{i_{t}}V(D_{k-1,n}^{q})]$ for $\{i_{1},i_{2},\ldots,i_{t}\}\subseteq\langle t_{k-1,n}+1\rangle,
k\geq 1$ and $t\geq 3$.
Then $\kappa (H)\geq 2$.
\end{lem}

\begin{proof}
For convenience, abbreviate $D_{k-1,n}^{i}$ as $D[i]$ for each $i\in \langle t_{k-1,n}+1\rangle$,
if there is no ambiguity.
There is a $k$-dimensional edge between any two distinct copies of $D_{k-1,n}$,
thus $H$ is connected.
Let $u$ and $v$ be any two distinct vertices of $H$.
To prove $\kappa (H)\geq 2$,
we only need to show that there are two internally disjoint $(u,v)$-paths in $H$.
\\
\indent
If $u$ and $v$ belong to the same copy of $D_{k-1,n}$,
without loss of generality, assume that $\{u,v\}\subseteq V(D[i_{1}])$.
By Lemma \ref{lem1}~$(1)$, $\kappa(D[i_{1}])=n+k-2$.
As $k\geq 1$ and $n\geq 3$, $\kappa(D[i_{1}])=n+k-2\geq 2$.
By Lemma \ref{lem6}, the result holds.
\\
\indent
If $u$ and $v$ belong to two distinct copies of $D_{k-1,n}$,
without loss of generality, assume that $u\in V(D[i_{1}])$ and $v\in V(D[i_{2}])$.
Let $H_{1}=D_{k,n}[\bigcup _{r=i_{3}}^{i_{t}}V(D_{k-1,n}^{r})]$.
Since there exists a unique edge between any two distinct copies of $D_{k-1,n}$.
Assume that $u_{1}u_{1}^{\prime}\in E(D[i_{1}],D[i_{2}])$,
$u_{2}u_{2}^{\prime}\in E(D[i_{1}],H_{1})$
and $v_{1}v_{1}^{\prime}\in E(D[i_{2}],H_{1})$,
where $u_{1},u_{2}\in V(D[i_{1}]),v_{1},u_{1}^{\prime}\in V(D[i_{2}])$
and $u_{2}^{\prime},v_{1}^{\prime}\in V(H_{1})$.
In addition, $t^{\prime}$ is the $k$-dimensional neighbor of $t$ with $t\in \{u_{1},u_{2},v_{1}\}$.
It is possible $u\in \{u_{1},u_{2}\}$ or $v\in \{v_{1},u_{1}^{\prime}\}$.
If $u\not\in \{u_{1},u_{2}\}$ and $v\not\in \{v_{1},u_{1}^{\prime}\}$,
by Lemma \ref{lem4}, there are $2$-fans from $u$ to $\{u_{1},u_{2}\}$ and
$v$ to $\{v_{1},u_{1}^{\prime}\}$ in $D[i_{1}]$ and $D[i_{2}]$, respectively.
If $u\in \{u_{1},u_{2}\}$ (or $v\in \{v_{1},u_{1}^{\prime}\}$),
there exists a path in $D[i_{1}]$ (or $D[i_{2}]$) between $u$ (or $v$) and $\{u_{1},u_{2}\}\backslash \{u\}$
(or $\{v_{1},u_{1}^{\prime}\}\backslash \{v\}$) as a connectivity of $D[i_{1}]$ (or $D[i_{2}]$).
Notice that there is a $(u_{2}^{\prime},v_{1}^{\prime})$-path in $H_{1}$ as $H_{1}$ is connected.
Then $u$ and $v$ are contained in a circuit.
Clearly, the result holds.
\end{proof}

\begin{thm}\label{thm2}
Let $D_{1,n}$ be the $1$-dimensional data center network with $n$-port switches for $n\geq 6$.
Then $\pi_{3}(D_{1,n})=\lfloor \frac{2n+3}{4}\rfloor$.
\end{thm}

\begin{proof}
Let $D_{1,n}=D_{0,n}^{0}\bigoplus D_{0,n}^{1}\bigoplus\ldots\bigoplus D_{0,n}^{t_{0,n}}$
with $D_{0,n}^{i}\cong D_{0,n}\cong K_{n}$ and $t_{0,n}=|V(D_{0,n})|=n$, where $i\in \langle n+1\rangle$.
For convenience, abbreviate $D_{0,n}^{i}$ as $D[i]$.
Let $u,v$ and $w$ be any three distinct vertices of $D_{0,n}$.
By Lemma \ref{lem1}~$(4)$, $|N_{D_{1,n}}(u)\cap N_{D_{1,n}}(v)\cap N_{D_{1,n}}(w)|= n-3$.
As $D_{1,n}$ is $n$-regular, by Lemma \ref{lem2}, $\pi_{3}(D_{1,n})\leq\lfloor \frac{3n-(n-3)}{4}\rfloor
=\lfloor \frac{2n+3}{4}\rfloor$.
To prove the result, we just need to show that $\pi_{D_{1,n}}(S)\geq \lfloor \frac{2n+3}{4}\rfloor$
for $S=\{x,y,z\}$, where $x,y$ and $z$ are any three distinct vertices of $D_{1,n}$.
Let $V(D[0])=\{u_{1},u_{2},\ldots,u_{n}\},V(D[1])=\{v_{1},v_{2},\ldots,
v_{n}\}$ and
$V(D[2])=\{w_{1},w_{2},\ldots,w_{n}\}$.
Consider the following three cases.

$\mathbf {Case~1}$. $x,y$ and $z$ belong to the same copy of $D_{0,n}$.

\indent
Without loss of generality, assume that $x,y,z\in V(D[0])$
and $\{x,y,z\}=\{u_{n-2},u_{n-1},u_{n}\}$.
By Lemma \ref{thm1} with $k=3$,
there are $\lfloor\frac{n}{2}\rfloor$ ID$S$-paths in $D_{0,n}$.
If $n\equiv 0($mod~$2)$, then $\lfloor\frac{n}{2}\rfloor=\lfloor \frac{2n+3}{4}\rfloor$,
which implies that the result holds.
If $n\equiv 1($mod~$2)$, then $\lfloor\frac{n}{2}\rfloor+1=\frac{n-1}{2}+1=\lfloor \frac{2n+3}{4}\rfloor$.
By Lemma \ref{lem1}~$(3)$, without loss of generality,
let $y^{\prime}\in V(D[1])$ and $z^{\prime}\in V(D[2])$,
where $y^{\prime}$ and $z^{\prime}$ are the $1$-dimensional neighbours of $y$ and $z$, respectively.
Notice that $D[1]$ and $D[2]$ are connected by the $1$-dimensional edge.
Then there exists a $(y^{\prime},z^{\prime})$-path, say $T$, in $D_{1,n}[V(D[1])\cup V(D[2])]$.
Let $P_{1}=xyz,P_{2}=xzu_{1}y,P_{\frac{i}{2}+2}=xu_{i}zu_{i+1}y$
and $P_{\frac{n-1}{2}+1}=xu_{n-3}yy^{\prime}T[y^{\prime},z^{\prime}]z^{\prime}z$
with $2\leq i\leq n-5$ and $i$ is even.
Then $P_{1},P_{2},\ldots,P_{\frac{n-1}{2}+1}$
are $\lfloor \frac{2n+3}{4}\rfloor$ ID$S$-paths in $D_{1,n}$.

\begin{figure}[ht]
\begin{center}
\scalebox{0.5}[0.5]{\includegraphics{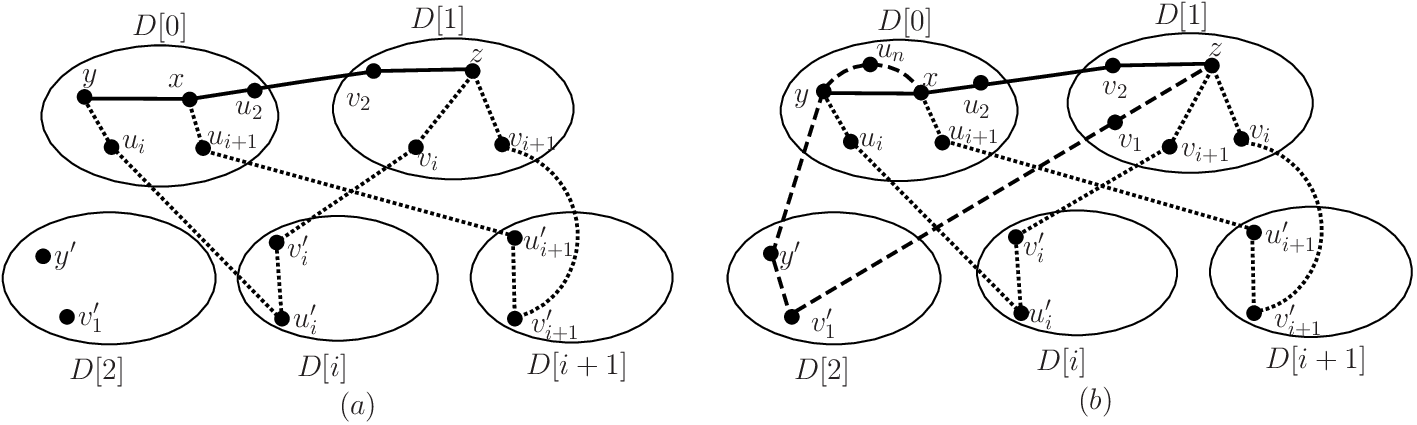}}\\
\captionsetup{font={small}}
\caption{The illustration of Case $2$ in Theorem \ref{thm2}}\label{Fig.2}
\end{center}
\end{figure}

$\mathbf {Case~2}$. $x,y$ and $z$ belong to two distinct copies of $D_{0,n}$.

\indent
Without loss of generality, assume that $x,y\in V(D[0])$ and $z\in V(D[1])$.
As $|\{x^{\prime},y^{\prime}\}\cap V(D[1])|\leq 1$,
without loss of generality, let $y^{\prime}\in V(D[2])$.
Suppose that $y=u_{1}$ and $x=u_{t}$ with some $t\in [n]$ and $t$ is even
(otherwise, relable the vertices in $V(D[0])$).
Since there exists a unique edge between any two distinct copies of $D_{k-1,n}$.
Assume that $u_{2}v_{2}\in E(D[0],D[1]),u_{i}u_{i}^{\prime}\in E(D[0],D[i]),v_{1}v_{1}^{\prime}\in E(D[1],D[2])$
and $v_{i}v_{i}^{\prime}\in E(D[1],D[i])$ with $i\in [n]\backslash \{1,2\}$.
By Lemma \ref{lem1}~$(2)$, one has that $y^{\prime}\neq v_{1}^{\prime}$ and $u_{i}^{\prime}\neq v_{i}^{\prime}$.
Let
$$P_{1}=yxu_{2}v_{2}z$$
and
$$P_{\frac{i+1}{2}}=yu_{i}u_{i}^{\prime}v_{i}^{\prime}
v_{i}zv_{i+1}v_{i+1}^{\prime}u_{i+1}^{\prime}u_{i+1}x,$$
where $3\leq i\leq n-1$ and $i$ is odd.
\\
\indent
Notice that there are some $i_{0},j_{0}\in [n]$ and $i$ is odd
such that $x=u_{i_{0}}$
and $z=v_{j_{0}}$.
If $i_{0}=2$ or $j_{0}=2$,
then $P_{1}^{\prime}$ is obtained from $P_{1}$ by replacing $xu_{2}$ with $x$, or $v_{2}z$ with $z$.
Otherwise, $P_{1}^{\prime}=P_{1}$.
If $i_{0}=i+1$ or $j_{0}=i$ or $j_{0}=i+1$ with $3\leq i\leq n-1$ and $i$ is odd,
then $P_{\frac{i+1}{2}}^{\prime}$ is obtained from $P_{\frac{i+1}{2}}$ by replacing
$u_{i+1}x$ with $x$, or $v_{i}z$ with $z$, or $zv_{i+1}$ with $z$.
Otherwise, $P_{\frac{i+1}{2}}^{\prime}=P_{\frac{i+1}{2}}$.
\\
\indent
If $n\equiv 0($mod~$2)$, then $\{P_{1}^{\prime},\ldots,P_{\frac{n}{2}}^{\prime}\}$
 is the set of
$\lfloor \frac{2n+3}{4}\rfloor$ ID$S$-paths
in $D_{1,n}$ (see Figure~\ref{Fig.2}~$(a)$).
\\
\indent
If $n\equiv 1($mod~$2)$,
let $P_{\frac{n-1}{2}+1}=xu_{n}yy^{\prime}v_{1}^{\prime}v_{1}z$.
If $j_{0}=1$, then $P_{\frac{n-1}{2}+1}^{\prime}$ is obtained from $P_{\frac{n-1}{2}+1}$
by replacing $v_{1}z$ with $z$.
Otherwise, $P_{\frac{n-1}{2}+1}^{\prime}=P_{\frac{n-1}{2}+1}$.
Thus $\{P_{1}^{\prime},\ldots,P_{\frac{n-1}{2}}^{\prime},P_{\frac{n-1}{2}+1}^{\prime}\}$
is the set of
$\lfloor \frac{2n+3}{4}\rfloor$ ID$S$-paths
in $D_{1,n}$ (see Figure~\ref{Fig.2}~$(b)$).

$\mathbf {Case~3}$. $x,y$ and $z$ belong to three distinct copies of $D_{0,n}$.

\indent
Without loss of generality, assume that $x\in V(D[0]),y\in V(D[1])$ and $z\in V(D[2])$.
By Lemma \ref{lem1}~$(3)$,
assume that $u_{1}v_{1}\in E(D[0],D[1]),u_{2}w_{1}\in E(D[0],D[2]),v_{2}w_{2}\in E(D[1],D[2])$,
$u_{j}u_{j}^{\prime}\in E(D[0],D[j]),v_{j}v_{j}^{\prime}\in E(D[1],D[j])$ and $w_{j}w_{j}^{\prime}\in E(D[2],D[j])$
with $3\leq j\leq n$.
Let
$$P_{1}=xu_{1}v_{1}yv_{2}w_{2}z;$$
$$P_{2}=zw_{1}u_{2}xu_{3}u_{3}^{\prime}v_{3}^{\prime}v_{3}y$$
and
$$P_{\frac{i}{2}+1}=zw_{i}w_{i}^{\prime}u_{i}^{\prime}
u_{i}xu_{i+1}u_{i+1}^{\prime}v_{i+1}^{\prime}v_{i+1}y,$$
where $4\leq i\leq n-1$ and $i$ is even (see Figure~\ref{Fig.3}).
\begin{figure}[ht]
\begin{center}
\scalebox{0.5}[0.5]{\includegraphics{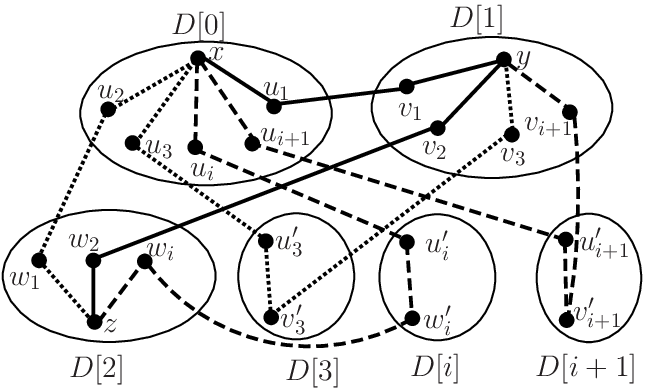}}\\
\captionsetup{font={small}}
\caption{The illustration of Case $3$ in Theorem \ref{thm2}}\label{Fig.3}
\end{center}
\end{figure}
\\
\indent
Recall that $x\in \{u_{1},u_{2},\ldots,u_{n}\}$,
$y\in \{v_{1},v_{2},\ldots,v_{n}\}$
and $z\in \{w_{1},w_{2},\ldots,w_{n}\}$.
Then there are some $i_{0},j_{0},k_{0}\in [n]$
such that $x=u_{i_{0}},y=v_{j_{0}}$ and $z=w_{k_{0}}$.
If $i_{0}=1$ or $j_{0}=1$ or $j_{0}=2$ or $k_{0}=2$,
then $P_{1}^{\prime}$ is obtained from $P_{1}$ by replacing $xu_{1}$ with $x$,
or $v_1y$ with $y$, or $yv_2$ with $y$, or $w_2z$ with $z$.
Otherwise, $P_1^{\prime}=P_1$.
If $i_0=2$ or $i_0=3$ or $j_0=3$ or $k_0=1$,
then $P_{2}^{\prime}$ is obtained from $P_{2}$ by replacing
$u_{2}x$ with $x$, or $xu_{3}$ with $x$, or $v_{3}y$ with $y$, or $zw_{1}$ with $z$.
Otherwise, $P_{2}^{\prime}=P_{2}$.
If $i_{0}=i$ or $i_{0}=i+1$ or $j_{0}=i+1$ or $k_{0}=i$ with $4\leq i\leq n-1$ and $i$ is even,
then $P_{\frac{i}{2}+1}^{\prime}$ is obtained from $P_{\frac{i}{2}+1}$ by replacing
$u_{i}x$ with $x$, or $xu_{i+1}$ with $x$, or $v_{i+1}y$ with $y$, or $zw_{i}$ with $z$.
Otherwise, $P_{\frac{i}{2}+1}^{\prime}=P_{\frac{i}{2}+1}$.
Thus $P_{1}^{\prime},P_{2}^{\prime},\ldots,P_{\frac{i}{2}+1}^{\prime}$ are
$\lfloor\frac{2n+3}{4}\rfloor$ ID$S$-paths in $D_{1,n}$.
\end{proof}

\begin{thm}\label{thm3}
Let $D_{k,n}$ be the $k$-dimensional data center network with $n$-port
switches for $k\geq 0$ and $n\geq 6$.
Then $\pi_{3}(D_{k,n})=\lfloor\frac{2n+3k}{4}\rfloor$.
\end{thm}

\begin{proof}
Let $D_{k,n}=D_{k-1,n}^{0}\bigoplus D_{k-1,n}^{1}\bigoplus\ldots\bigoplus D_{k-1,n}^{t_{k-1,n}}$
with $D_{k-1,n}^{i}\cong D_{k-1,n}$ and $t_{k-1,n}=|V(D_{k-1,n})|$, where $i\in \langle t_{k-1,n}+1\rangle$.
For convenience, abbreviate $D_{k-1,n}^{i}$ as $D[i]$
and define $u^{\prime}$ as the $k$-dimensional neighbour of each vertex $u$ in $D_{k,n}$.
\\
\indent
For a given $n$ with $n\geq 4$, we apply induction on $k$ to prove the result.
The result is true for $k=0$ and $k=1$
from Lemma \ref{thm1} and Theorem \ref{thm2}, respectively.
Suppose that the result holds for $D_{k-1,n}$ with $k\geq 2$.
Consider $D_{k,n}$ as follows.
As $D_{k,n}$ is $(n+k-1)$-regular, by Lemma \ref{lem1}~$(4)$ and Lemma \ref{lem2},
$\pi_{3}(D_{k,n})\leq\lfloor \frac{2n+3k}{4}\rfloor$.
Thus it suffices to prove that
there exist at least $\lfloor\frac{2n+3k}{4}\rfloor$ internally disjoint paths
connecting any three distinct vertices in $D_{k,n}$.
Let $S=\{x,y,z\}$,
where $x,y$ and $z$ are any three distinct vertices in $D_{k,n}$.
The following three claims hold.

\begin{Claim}
There exist at least $ \lfloor \frac{2n+3k}{4}\rfloor$ ID$S$-paths in $D_{k,n}$,
if the vertices of $S$ belong to the same copy of $D_{k-1,n}$.
\end{Claim}


\noindent {\it Proof of Claim 1.}
Without loss of generality, assume that $S\subseteq V(D[0])$.
By the inductive hypothesis,
$\pi_{3}(D[0])=\lfloor\frac{2n+3(k-1)}{4}\rfloor$.
Then
$\pi_{D[0]}(S)\geq \lfloor\frac{2n+3(k-1)}{4}\rfloor$.
\\
\indent
If
$n\equiv 0($mod~$2)$ and $k\equiv 1($mod~$4)$, or $n\equiv 1($mod~$2)$
and $k\equiv 3($mod~$4)$,
then
$\lfloor\frac{2n+3(k-1)}{4}\rfloor=\lfloor\frac{2n+3k}{4}\rfloor$.
Thus $\pi_{D_{k,n}}(S)\geq \pi_{D[0]}(S)\geq\lfloor\frac{2n+3(k-1)}{4}\rfloor=\lfloor\frac{2n+3k}{4}\rfloor$.
\begin{figure}[ht]
\begin{center}
\scalebox{0.5}[0.5]{\includegraphics{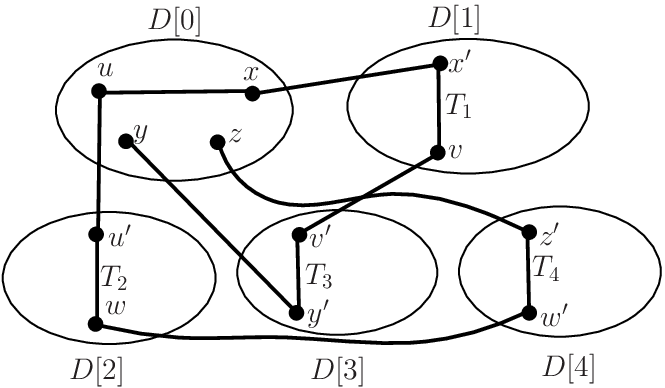}}\\
\captionsetup{font={small}}
\caption{The illustration of Claim $1$ in Theorem \ref{thm3}}\label{Fig.4}
\end{center}
\end{figure}
\\
\indent
If $n\equiv 0($mod~$2)$ and $k\not\equiv 1($mod~$4)$,
or $n\equiv 1($mod~$2)$ and $k\not\equiv 3($mod~$4)$,
then
$\lfloor\frac{2n+3(k-1)}{4}\rfloor+1=\lfloor\frac{2n+3k}{4}\rfloor$.
Recall that $\pi_{3}(D[0])=\lfloor\frac{2n+3(k-1)}{4}\rfloor$.
Then $d_{D[0]}(x)+d_{D[0]}(y)+d_{D[0]}(z)-4\pi_{3}(D[0])-(n-3)
=3(n+k-2)-4\times \lfloor\frac{2n+3(k-1)}{4}\rfloor-(n-3)\geq 1$ as
$D_{k-1,n}$ is $(n+k-2)$-regular.
By Lemma \ref{lem1}~($4$) and Lemma \ref{thm21}~(1),
then there exist at least one vertex, say $u$,
and a set $\mathcal{P}$ of $\lfloor\frac{2n+3(k-1)}{4}\rfloor$ ID$S$-paths
such that $u\in (N_{D[0]}(x)\cup N_{D[0]}(y)\cup N_{D[0]}(z))\backslash V(\mathcal{P})$.
As $\lfloor\frac{2n+3(k-1)}{4}\rfloor+1=\lfloor\frac{2n+3k}{4}\rfloor$,
we need to find one more desired path except the paths in $\mathcal{P}$.
Without loss of generality, assume that $u\in N_{D[0]}(x)\backslash V(\mathcal{P})$.
According to Lemma \ref{lem1} ($2$) and Lemma \ref{lem1} $(3)$,
assume that $x^{\prime}\in V(D[1]),u^{\prime}\in V(D[2]),
y^{\prime}\in V(D[3])$ and $z^{\prime}\in V(D[4])$.
Let $v\in V(D[1])$ and $w\in V(D[2])$ such that $v^{\prime}\in V(D[3])$ and $w^{\prime}\in V(D[4])$.
Since $D[i]$ is connected, there exists one path, denoted by $T_{i}$
connecting $\alpha$ and $\beta$ in $D[i]$,
where $\alpha\in\{x^{\prime},u^{\prime},y^{\prime},z^{\prime}\}\cap V(D[i]),\beta\in\{v,w,v^{\prime},w^{\prime}\}\cap V(D[i])$
and $1\leq i\leq 4$.
Let $P=zz^{\prime}T_{4}[z^{\prime},w^{\prime}]w^{\prime}w
T_{2}[w,u^{\prime}]u^{\prime}uxx^{\prime}T_{1}[x^{\prime},v]vv^{\prime}T_{3}[v^{\prime},y^{\prime}]y^{\prime}y$
 (see Figure~\ref{Fig.4}).
Then $\mathcal{P}\cup \{P\}$ is the set of
$\lfloor\frac{2n+3k}{4}\rfloor$ ID$S$-paths in $D_{k,n}$.
$\hfill\blacksquare$

\indent
Let $u$ be a vertex of $D_{k-1,n}$ and
$A$ be a set of $n+k-2$ vertices of $D_{k-1,n}$ such that $u\not\in A$.
By Fan Lemma \ref{lem4},
there is a fan, say $\mathcal{F}$, in $D_{k-1,n}$ from $u$ to $A$.
For any vertex $v\in A$, we denote \emph{$\mathcal{F}[u,v]$} (or \emph{$\mathcal{F}[v,u]$})
\emph{the $(u,v)$-path} in $\mathcal{F}$.
\begin{Claim}
There exist at least $ \lfloor \frac{2n+3k}{4}\rfloor$ ID$S$-paths in $D_{k,n}$,
if the vertices of $S$ belong to three distinct copies of $D_{k-1,n}$.
\end{Claim}

\noindent {\it Proof of Claim 2.}
Without loss of generality, assume that $x\in V(D[0]),y\in V(D[1])$ and $z\in V(D[2])$.
Since $V(D[0])$ is the union of copies of $V(D_{k-2,n})$,
assume that $x,u\in V(D_{k-2,n}^{0})$ such that $N_{D_{k-2,n}^{0}}(u)\cap N_{D_{k-2,n}^{0}}(x)\neq \emptyset$.
Let $v\in V(D[0])\backslash V(D_{k-2,n}^{0})$ such that
$v^{\prime}\in N_{D_{k-2,n}^{0}}(u)\cap N_{D_{k-2,n}^{0}}(x)$,
where $v^{\prime}$ is the $(k-1)$-dimensional neighbour of $v$ in $D[0]$.
Thus $|N_{D[0]}(x)\cap N_{D[0]}(u)\cap N_{D[0]}(v)|=\gamma=1$ as Lemma \ref{lem1}~(2).
By the inductive hypothesis, $\pi_{3}(D[0])=\lfloor\frac{2n+3(k-1)}{4}\rfloor$.
When $n\equiv 0($mod~$2)$ and $k\not\equiv 1($mod~$4)$,
or $n\equiv 1($mod~$2)$ and $k \not\equiv 3($mod~$4)$,
then $d_{D[0]}(x)+d_{D[0]}(u)+d_{D[0]}(v)-4\pi_{3}(D[0])-\gamma
=3(n+k-2)-4\times \lfloor\frac{2n+3(k-1)}{4}\rfloor-1\geq 3$.
By Lemma \ref{thm21}~(2),
there exist two distinct vertices $\alpha,\beta$ and a set $\mathcal{P}$ of
$\lfloor\frac{2n+3(k-1)}{4}\rfloor$
internally disjoint $\{x,u,v\}$-paths in $D[0]$
such that such that
$\alpha,\beta\in ((N_{D[0]}(x)\cup N_{D[0]}(u)\cup N_{D[0]}(v))\backslash V(\mathcal{P}))$.
It implies that
$|(N_{D[0]}(x)\cup N_{D[0]}(u)\cup N_{D[0]}(v))\backslash V(\mathcal{P})|\geq 2$.
Denoted by the set of these paths $\mathcal{P}=\{T_{1},T_{2},\ldots,T_{r},$
$\tilde{T}_{1},\tilde{T}_{2},\ldots,\tilde{T}_{s},
\hat{T}_{1},\hat{T}_{2},\ldots,\hat{T}_{t}\}$,
where $r+s+t=\lfloor\frac{2n+3(k-1)}{4}\rfloor$.
Without loss of generality, assume that $d_{T_{i}}(u)=d_{\tilde{T}_{j}}(v)=d_{\hat{T}_{\ell}}(x)=2$
with $i\in [r],j\in [s]$ and $\ell\in [t]$.
Then $d_{\mathcal{P}}(x)=r+s+2t\leq d_{D[0]}(x)=n+k-2, d_{\mathcal{P}}(u)=2r+s+t\leq d_{D[0]}(u)=n+k-2$
and $d_{\mathcal{P}}(v)=r+2s+t\leq d_{D[0]}(v)=n+k-2$.
\\
\indent
Recall that $\alpha^{\prime}$ is the $k$-dimensional neighbour of a vertex $\alpha$ in $D_{k,n}$.
Select $r+s+2t$ vertices, say $x_{1i},x_{2j},x_{3\ell},\bar{x}_{3\ell}$, different from $x$ in $D[0]$;
$2r+s+t$ vertices, say $y_{1i},\bar{y}_{1i},y_{2j},y_{3\ell}$, different from $y$ in $D[1]$
and $r+2s+t$ vertices, say $z_{1i},z_{2j},\bar{z}_{2j},z_{3\ell}$, different from $z$ in $D[2]$
such that the condition $\bm{(\mathcal{A})}$ holds,
where
$\bm{(\mathcal{A})}:$ {\bf any two vertices in each of $\bm{\{z^{\prime},x_{1i}^{\prime},x_{3\ell}^{\prime}\}}$,
$\bm{\{x^{\prime},\bar{y}_{1i}^{\prime},y_{2j}^{\prime}\}}$ and $\bm{\{y^{\prime},x_{2j}^{\prime},\bar{x}_{3\ell}^{\prime}\}}$
belong to distinct copies of $\bm{D_{k-1,n}}$,
where $\bm{1\leq i\leq r,1\leq j\leq s}$ and $\bm{1\leq \ell\leq t}$}.
Without loss of generality,
assume that
$\{x_{1i}^{\prime},y_{1i}^{\prime}\}\subseteq V(D[i+2]),
\{\bar{y}_{1i}^{\prime},z_{1i}^{\prime}\}\subseteq V(D[r+i+2])$,
$\{x_{2j}^{\prime},z_{2j}^{\prime}\}\subseteq V(D[2r+j+2]),
\{y_{2j}^{\prime},\bar{z}_{2j}^{\prime}\}\subseteq V(D[2r+s+j+2])$,
$\{x_{3\ell}^{\prime},y_{3\ell}^{\prime}\}\subseteq V(D[2r+2s+\ell+2])$
and $\{\bar{x}_{3\ell}^{\prime},z_{3\ell}^{\prime}\}\subseteq V(D[2r+2s+t+\ell+2])$,
where $1\leq i\leq r,1\leq j\leq s$ and $1\leq \ell\leq t$.
This can be done for the following reasons:
When $k\geq 2$ and $n\geq 6$,
the number $t_{k-1,n}+1$ of the copies of $D_{k-1,n}$ in $D_{k,n}$
is much greater than $2(n+k-1)$ and $2r+2s+2t+2\leq 2(n+k-1)$.
Let $\bm{H_{0}=D_{k,n}[\bigcup_{i=3}^{2r+2s+2t+2}V(D[i])]}$,
by Lemma \ref{lem1}~$(3)$ and $\bm{(\mathcal{A})}$, one has that
$x^{\prime},y^{\prime},z^{\prime}\not\in V(H_{0})$.
\\
\indent
Notice that $D[\delta]$ is connected with $\delta\in \langle t_{k-1,n}+1\rangle$.
Then there exists one path connecting any two distinct vertices in $D[\delta]$.
Denoted by $U_{1i},\bar{U}_{1i},U_{2j},\bar{U}_{2j},U_{3\ell}$ and $\bar{U}_{3\ell}$
the $(x_{1i}^{\prime},y_{1i}^{\prime})$-path, $(\bar{y}_{1i}^{\prime},z_{1i}^{\prime})$-path,
$(x_{2j}^{\prime},z_{2j}^{\prime})$-path, $(y_{2j}^{\prime},\bar{z}_{2j}^{\prime})$-path,
$(x_{3\ell}^{\prime},y_{3\ell}^{\prime})$-path and $(\bar{x}_{3\ell}^{\prime},z_{3\ell}^{\prime})$-path
in $D[i+2],D[r+i+2],D[2r+j+2],D[2r+s+j+2],D[2r+2s+\ell+2]$ and $D[2r+2s+t+\ell+2]$, respectively.
\\
\indent
Let
$$A_{1}=\{x_{1i},x_{2j},x_{3\ell},\bar{x}_{3\ell}|1\leq i\leq r,1\leq j\leq s,1\leq \ell\leq t\},$$
$A_{2}\subseteq V(D[0])\backslash (A_{1}\cup \{x\})$
and $A=A_{1}\cup A_{2}$ such that $|A|=|A_{1}|+|A_{2}|=n+k-2$;
$$B_{1}=\{y_{1i},\bar{y}_{1i},y_{2j},y_{3\ell}|1\leq i\leq r,1\leq j\leq s,1\leq \ell\leq t\},$$
$B_{2}\subseteq V(D[1])\backslash (B_{1}\cup \{y\})$
and $B=B_{1}\cup B_{2}$ such that $|B|=|B_{1}|+|B_{2}|=n+k-2$;
$$C_{1}=\{z_{1i},z_{2j},\bar{z}_{2j},z_{3\ell}|1\leq i\leq r,1\leq j\leq s,1\leq \ell\leq t\},$$
$C_{2}\subseteq V(D[2])\backslash (C_{1}\cup \{z\})$
and $C=C_{1}\cup C_{2}$ such that $|C|=|C_{1}|+|C_{2}|=n+k-2$.
If $|\Lambda_{1}|=n+k-2$, then $\Lambda_{2}=\emptyset$,
where $\Lambda\in\{A,B,C\}$.
Otherwise, $\Lambda_{2}\neq\emptyset$.
\\
\indent
By Fan Lemma \ref{lem4}, there is a fan $\mathcal{F}$ which contains $n+k-2$ internally disjoint paths
from $x$ to $A$ in $D[0]$.
Similarly,
there are two fans, denoted by $\mathcal{F}_{1}$ and $\mathcal{F}_{2}$ from $y$ to $B$
and from $z$ to $C$ in $D[1]$ and $D[2]$, respectively.
Let
$$
\begin{aligned}
P_{i}=\mathcal{F}[x,x_{1i}]x_{1i}x_{1i}^{\prime}U_{1i}[x_{1i}^{\prime},y_{1i}^{\prime}]
y_{1i}^{\prime}y_{1i}\mathcal{F}_{1}[y_{1i},y]y\mathcal{F}_{1}[y,\bar{y}_{1i}]\bar{y}_{1i}\bar{y}_{1i}^{\prime}
\bar{U}_{1i}[\bar{y}_{1i}^{\prime},z_{1i}^{\prime}]z_{1i}^{\prime}z_{1i}\mathcal{F}_{2}[z_{1i},z]~&(1),\\
\tilde{P}_{j}=\mathcal{F}[x,x_{2j}]x_{2j}x_{2j}^{\prime}U_{2j}[x_{2j}^{\prime},z_{2j}^{\prime}]
z_{2j}^{\prime}z_{2j}\mathcal{F}_{2}[z_{2j},z]z\mathcal{F}_{2}[z,\bar{z}_{2j}]
\bar{z}_{2j}\bar{z}_{2j}^{\prime}\bar{U}_{2j}[\bar{z}_{2j}^{\prime},y_{2j}^{\prime}]y_{2j}^{\prime}y_{2j}
\mathcal{F}_{1}[y_{2j},y]~&(2),\\
\hat{P}_{\ell}=\mathcal{F}_{1}[y,y_{3\ell}]y_{3\ell}y_{3\ell}^{\prime}U_{3\ell}[y_{3\ell}^{\prime},x_{3\ell}^{\prime}]
x_{3\ell}^{\prime}x_{3\ell}\mathcal{F}[x_{3\ell},x]x\mathcal{F}[x,\bar{x}_{3\ell}]
\bar{x}_{3\ell}\bar{x}_{3\ell}^{\prime}\bar{U}_{3\ell}[\bar{x}_{3\ell}^{\prime},z_{3\ell}^{\prime}]z_{3\ell}^{\prime}z_{3\ell}
\mathcal{F}_{2}[z_{3\ell},z]~&(3)
\end{aligned}
$$
with $1\leq i\leq r,1\leq j\leq s$ and $1\leq\ell\leq t$.

$\mathbf {Part~I}$. $n\equiv 0($mod~$2)$ and $k\equiv 1($mod~$4)$,
or $n\equiv 1($mod~$2)$ and $k\equiv 3($mod~$4)$.

When $n\equiv 0($mod~$2)$ and $k\equiv 1($mod~$4)$,
or $n\equiv 1($mod~$2)$ and $k\equiv 3($mod~$4)$,
then $\lfloor\frac{2n+3(k-1)}{4}\rfloor=\lfloor\frac{2n+3k}{4}\rfloor$.
Thus $P_{1},\ldots,P_{r},\tilde{P}_{1},\ldots,\tilde{P}_{s},\hat{P}_{1},\ldots,\hat{P}_{t}$
are our desired ID$S$-paths in $D_{k,n}$ (see Figure~\ref{Fig.9}).

\begin{figure}[ht]
\begin{center}
\scalebox{0.5}[0.5]{\includegraphics{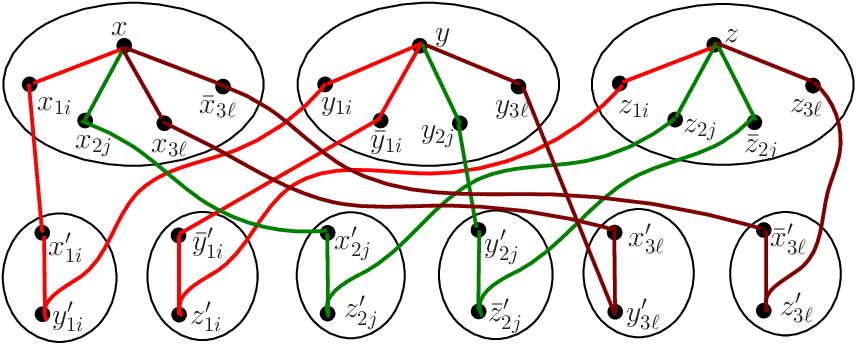}}\\
\captionsetup{font={small}}
\caption{}\label{Fig.9}
\end{center}
\end{figure}

$\mathbf {Part~II}$. $n\equiv 0($mod~$2)$ and $k\not\equiv 1($mod~$4)$,
or $n\equiv 1($mod~$2)$ and $k \not\equiv 3($mod~$4)$.

When $n\equiv 0($mod~$2)$ and $k\not\equiv 1($mod~$4)$,
or $n\equiv 1($mod~$2)$ and $k \not\equiv 3($mod~$4)$,
then $\lfloor\frac{2n+3(k-1)}{4}\rfloor+1=\lfloor\frac{2n+3k}{4}\rfloor$.
Recall that
$|(N_{D[0]}(x)\cup N_{D[0]}(u)\cup N_{D[0]}(v))\backslash V(\mathcal{P})|\geq 2$ as Lemma \ref{thm21}~(2).
Let $\bm{H=D_{k,n}[\bigcup_{q=2r+2s+2t+3}^{t_{k-1,n}}V(D[q])]}$.
Then $H$ is $2$-connected as Lemma \ref{lem5}.
\\
\indent
According to the locations of the $k$-dimensional neighbours $x^{\prime},y^{\prime}$ and $z^{\prime}$ of $x,y$ and $z$,
the following cases are considered:
$\mathbf {Case~1}$: $\{x^{\prime},y^{\prime},z^{\prime}\}\subseteq V(H)$.
Each of the remaining cases has at least one vertex in $\{x^{\prime},y^{\prime},z^{\prime}\}$,
say $y^{\prime}$, not in $H$.
Since $y^{\prime}\not\in V(H_{0})$, where $H_{0}=D_{k,n}[\bigcup_{i=3}^{2r+2s+2t+2}V(D[i])]$,
one has that $y^{\prime}\in V(D[0])\cup V(D[2])$.
Without loss of generality,
assume that $y^{\prime}\in V(D[0])$.
Furthermore, we consider whether $x^{\prime}$ is in $D[2]$.
Finally, if $x^{\prime}\not\in V(D[2])$, we consider whether $z^{\prime}$ is in $D[0]\cup D[1]$.
As a result, the following three cases need to be considered except Case $1$. That is
$\mathbf {Case~2}$: $y^{\prime}\in V(D[0]),x^{\prime}\not\in V(D[2])$ and $z^{\prime}\not\in V(D[0]\cup D[1])$.
$\mathbf {Case~3}$: $y^{\prime}\in V(D[0]),x^{\prime}\not\in V(D[2])$ and $z^{\prime}\in V(D[0]\cup D[1])$.
And $\mathbf {Case~4}$: $y^{\prime}\in V(D[0])$ and $x^{\prime}\in V(D[2])$.

\indent
$\mathbf {Case~1}$. $\{x^{\prime},y^{\prime},z^{\prime}\}\subseteq V(H)$.

Recall that $(N_{D[0]}(x)\cup N_{D[0]}(u)\cup N_{D[0]}(v))\backslash V(\mathcal{P})\neq \emptyset$,
without loss of generality, assume that $N_{D[0]}(x)\backslash V(\mathcal{P})\neq \emptyset$.
It implies that $r+s+2t=d_{\mathcal{P}}(x)< d_{D[0]}(x)=n+k-2$.
Notice that $|A_{1}|=d_{\mathcal{P}}(x)=r+s+2t$
and $|A|=|A_{1}|+|A_{2}|=n+k-2$.
Choose $x_{1}$ in $D[0]$ such that
$x_{1}^{\prime}\in V(H)$ and let $x_{1}\in A_{2}$.
Let $D_{1}=\{a,b\}=\{x_{1}^{\prime},x^{\prime}\}$
and $D_{2}=\{c,d\}=\{y^{\prime},z^{\prime}\}$.
Then there are two internally disjoint $(D_{1},D_{2})$-paths, say $\mathcal{D}_{1}[c,a]$
and $\mathcal{D}_{2}[b,d]$, in $H$ as $H$ is 2-connected.
Let $\hat{P}_{t+1}=yy^{\prime}\mathcal{D}_{1}[y^{\prime},x_{1}^{\prime}]x_{1}^{\prime}x_{1}\mathcal{F}[x_{1},x]
xx^{\prime}\mathcal{D}_{2}[x^{\prime},z^{\prime}]z^{\prime}z$ if $a=x_{1}^{\prime}$ and $c=y^{\prime}$
(see Figure~\ref{Fig.10}$~(a)$).
Otherwise, $\hat{P}_{t+1}=yy^{\prime}\mathcal{D}_{1}[y^{\prime},x^{\prime}]x^{\prime}x\mathcal{F}[x,x_{1}]
x_{1}x_{1}^{\prime}
\mathcal{D}_{2}[x_{1}^{\prime},z^{\prime}]z^{\prime}z$.
Then $P_{1},\ldots,P_{r},\tilde{P}_{1},\ldots,\tilde{P}_{s},\hat{P}_{1},$
$\ldots,\hat{P}_{t}$ and $\hat{P}_{t+1}$
are our desired ID$S$-paths in $D_{k,n}$.

\begin{figure}[ht]
\begin{center}
\scalebox{0.5}[0.5]{\includegraphics{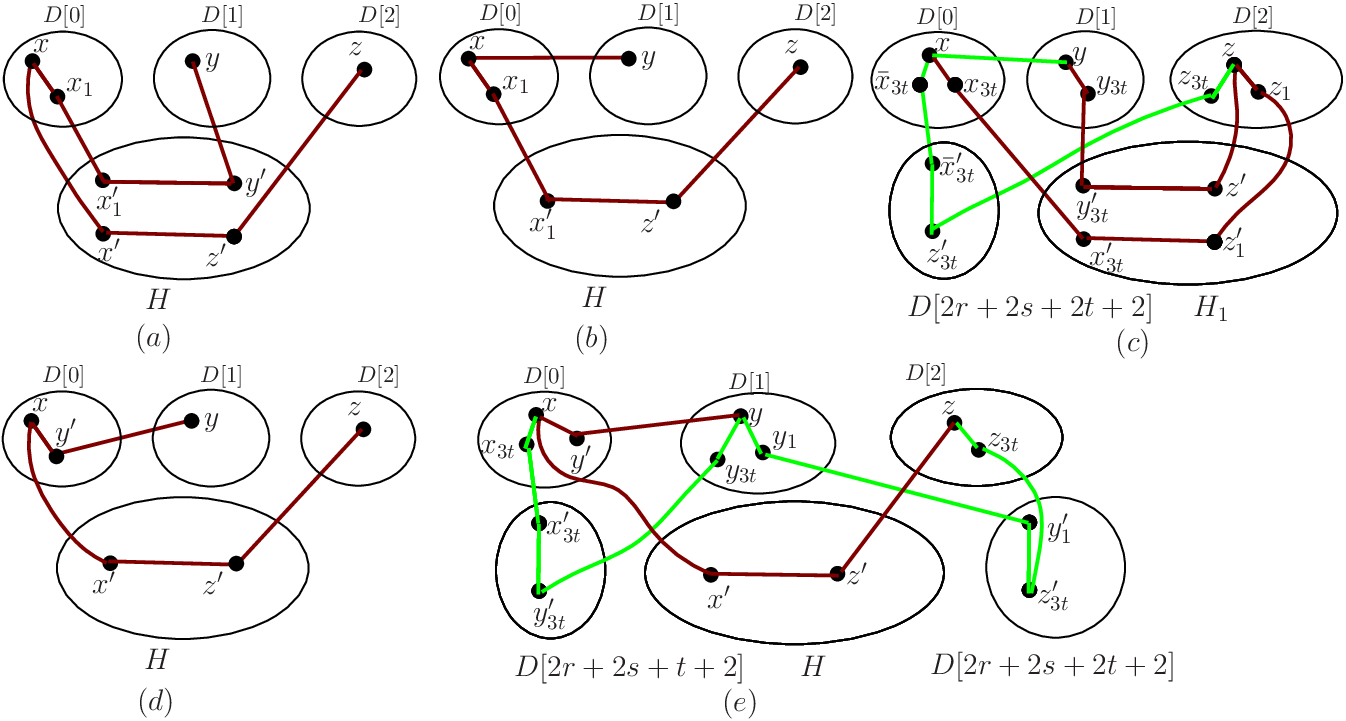}}\\
\captionsetup{font={small}}
\caption{The illustrations of Case~$1$ and Case~$2$ in Claim~2}\label{Fig.10}
\end{center}
\end{figure}

$\mathbf {Case~2}$. $y^{\prime}\in V(D[0]),x^{\prime}\not\in V(D[2])$ and $z^{\prime}\not\in V(D[0]\cup D[1])$.

As $z^{'}\not \in V(D[0]\cup D[1]\cup H_{0})$,
by Lemma \ref{lem1}~$(3)$,
$z^{\prime}\in V(H)$.

$\mathbf {Case~2.1}$. $x=y^{\prime}$.

\indent
Recall that $(N_{D[0]}(x)\cup N_{D[0]}(u)\cup N_{D[0]}(v))\backslash V(\mathcal{P})\neq \emptyset$.
If $(N_{D[0]}(x)\cup N_{D[0]}(u))\backslash V(\mathcal{P})\neq \emptyset$,
without loss of generality, assume that $N_{D[0]}(x)\backslash V(\mathcal{P})\neq \emptyset$.
Then $r+s+2t=|A_{1}|=d_{\mathcal{P}}(x)< n+k-2$.
Choose $x_{1}$ in $D[0]$ such that
$x_{1}^{\prime}\in V(H)$ and let $x_{1}\in A_{2}$.
Then there is a $(x_{1}^{\prime},z^{\prime})$-path, say $L[x_{1}^{\prime},z^{\prime}]$, in $H$.
Let $\hat{P}_{t+1}=yx\mathcal{F}[x,x_{1}]x_{1}x_{1}^{\prime}L[x_{1}^{\prime},z^{\prime}]z^{\prime}z$
(see Figure~\ref{Fig.10}$~(b)$).
Then $P_{1},\ldots,P_{r},\tilde{P}_{1},\ldots,\tilde{P}_{s},\hat{P}_{1},\ldots,\hat{P}_{t}$ and $\hat{P}_{t+1}$
are our desired ID$S$-paths in $D_{k,n}$.
\\
\indent
If $N_{D_{0}}(v)\backslash V(\mathcal{P})\neq\emptyset$,
then $|C_{1}|< n+k-2$.
Choose $z_{1}$ in $D[2]$
such that $z_{1}^{\prime}\in V(D[2r+2s+t+2])$ and let $z_{1}\in C_{2}$.
We will delete the path $\hat{P}_{t}$ and find two other paths.
Let $H_{1}=D_{k,n}[V(H)\cup V(D[2r+2s+t+2])]$.
Then $H_{1}$ is $2$-connected as Lemma \ref{lem5}.
Let $D_{3}=\{a,b\}=\{y_{3t}^{\prime},x_{3t}^{\prime}\}$
and $D_{4}=\{c,d\}=\{z^{\prime},z_{1}^{\prime}\}$.
Then there are two internally disjoint $(D_{3},D_{4})$-paths,
say $\mathcal{D}_{3}[c,a]$ and $\mathcal{D}_{4}[b,d]$, in $H_{1}$.
Let $\hat{P}_{t+1}=yx\mathcal{F}[x,\bar{x}_{3t}]\bar{x}_{3t}\bar{x}_{3t}^{\prime}
\bar{U}_{3t}[\bar{x}_{3t}^{\prime},z_{3t}^{\prime}]z_{3t}^{\prime}z_{3t}\mathcal{F}_{2}[z_{3t},z]$
and $\hat{P}_{t+2}=\mathcal{F}_{1}[y,y_{3t}]y_{3t}y_{3t}^{\prime}\mathcal{D}_{4}[y_{3t}^{\prime},z^{\prime}]
z^{\prime}z
\mathcal{F}_{2}[z,z_{1}]z_{1}$$z_{1}^{\prime}$$\mathcal{D}_{3}[z_{1}^{\prime},x_{3t}^{\prime}]
x_{3t}^{\prime}x_{3t}\mathcal{F}[x_{3t},x]$ if $a=x_{3t}^{\prime}$ and $c=z_{1}^{\prime}$ (see Figure~\ref{Fig.10}$~(c)$).
Otherwise, $\hat{P}_{t+2}=\mathcal{F}_{1}[y,y_{3t}]y_{3t}y_{3t}^{\prime}\mathcal{D}_{4}[y_{3t}^{\prime},
z_{1}^{\prime}]
z_{1}^{\prime}z_{1}\mathcal{F}_{2}[z_{1},z]zz^{\prime}
\mathcal{D}_{3}[z^{\prime},x_{3t}^{\prime}]x_{3t}^{\prime}x_{3t}$ $\mathcal{F}[x_{3t},x]$.
Then $P_{1},\ldots,P_{r},\tilde{P}_{1},\ldots,\tilde{P}_{s},\hat{P}_{1},\ldots,\hat{P}_{t-1},\hat{P}_{t+1}$
and $\hat{P}_{t+2}$
are our desired ID$S$-paths in $D_{k,n}$.

$\mathbf {Case~2.2}$. $x\neq y^{\prime}$.

Since $y^{\prime}\in V(D[0])$ and $x\neq y^{\prime}$,
$x^{\prime}\not\in V(D[1])$ as Lemma \ref{lem1}~$(3)$.
Notice that $x^{\prime}\not\in V(D[1]\cup D[2]\cup H_{0})$.
Then $x^{\prime}\in V(H)$.
Since $H$ is connected, there is a $(x^{\prime},z^{\prime})$-path, say $L[x^{\prime},z^{\prime}]$, in $H$.
Recall that $(N_{D[0]}(x)\cup N_{D[0]}(u)\cup N_{D[0]}(v))\backslash V(\mathcal{P})\neq\emptyset$.
If $N_{D[0]}(x)\backslash V(\mathcal{P})\neq\emptyset$,
then $|A_{1}|< n+k-2$.
Select $x_{1}$ in $D[0]$ such that $x_{1}=y^{\prime}$ and let $x_{1}\in A_{2}$.
Let $\hat{P}_{t+1}=yy^{\prime}\mathcal{F}[y^{\prime},x]xx^{\prime}L[x^{\prime},z^{\prime}]z^{\prime}z$
(see Figure~\ref{Fig.10}$~(d)$).
Then $P_{1},\ldots,P_{r},\tilde{P}_{1},\ldots,\tilde{P}_{s},\hat{P}_{1},\ldots,\hat{P}_{t}$
and $\hat{P}_{t+1}$
are our desired ID$S$-paths in $D_{k,n}$.
\\
\indent
If $(N_{D[0]}(u)\cup N_{D[0]}(v))\backslash V(\mathcal{P})\neq\emptyset$,
without loss of generality, assume that $N_{D[0]}(u)\backslash V(\mathcal{P})\neq\emptyset$.
Then $2r+s+t=|B_{1}|<n+k-2$.
We obtain the desired $S$-paths in $D_{k,n}$ by deleting the path $\hat{P}_{t}$ and constructing
 the two paths $\hat{P}_{t+1}$ and $\hat{P}_{t+2}$.
Choose $y_{1}$ in $D[1]$ such that $y_{1}^{\prime}\in V(D[2r+2s+2t+2])$ and let $y_{1}\in B_{2}$.
Let $\bar{U}_{3t}^{\prime}$ be the $(y_{1}^{\prime},z_{3t}^{\prime})$-path in $D[2r+2s+2t+2]$
as $D[2r+2s+2t+2]$ is connected
and
let $A^{\prime}=A\backslash\{\bar{x}_{3t}\}\cup \{y^{\prime}\}$.
By Fan Lemma \ref{lem4},
there is a fan, say $\mathcal{F}^{\prime}$, in $D[0]$ from $x$ to $A^{\prime}$.
Let $P_{i}^{\prime}$ be the path obtained from $P_{i}$ shown in formula $(1)$ by replacing $\mathcal{F}[x,x_{1i}]$
with $\mathcal{F}^{\prime}[x,x_{1i}]$ for $1\leq i\leq r$;
$\tilde{P}_{j}^{\prime}$ be the path obtained from $\tilde{P}_{j}$ shown in formula $(2)$
by replacing $\mathcal{F}[x,x_{2j}]$ with $\mathcal{F}^{\prime}[x,x_{2j}]$ for $1\leq j \leq s$;
$\hat{P}_{f}^{\prime}$ be the path obtained from $\hat{P}_{f}$ shown in formula $(3)$
by replacing $\mathcal{F}[x_{3f},x]$ and $\mathcal{F}[x,\bar{x}_{3f}]$
with $\mathcal{F}^{\prime}[x_{3f},x]$ and $\mathcal{F}^{\prime}[x,\bar{x}_{3f}]$,
respectively, for $1\leq f \leq t-1$.
Let $\hat{P}_{t+1}=\mathcal{F}^{\prime}[x,x_{3t}]x_{3t}x_{3t}^{\prime}U_{3t}[x_{3t}^{\prime},y_{3t}^{\prime}]
y_{3t}^{\prime}y_{3t}\mathcal{F}_{1}[y_{3t},y]y\mathcal{F}_{1}[y,y_{1}]y_{1}y_{1}^{\prime}
\bar{U}_{3t}^{\prime}[y_{1}^{\prime},z_{3t}^{\prime}]z_{3t}^{\prime}z_{3t}\mathcal{F}_{2}[z_{3t},z]$
and $\hat{P}_{t+2}=yy^{\prime}\mathcal{F}^{\prime}[y^{\prime},x]xx^{\prime}L[x^{\prime},z^{\prime}]z^{\prime}z$
(see Figure~\ref{Fig.10}$~(e)$).
Then $P_{1}^{\prime},\ldots,P_{r}^{\prime},\tilde{P}_{1}^{\prime},\ldots,\tilde{P}_{s}^{\prime},
\hat{P}_{1}^{\prime},\ldots,$ $\hat{P}_{t-1}^{\prime},\hat{P}_{t+1}$
and $\hat{P}_{t+2}$
are our desired ID$S$-paths in $D_{k,n}$.

$\mathbf {Case~3}$. $y^{\prime}\in V(D[0]),x^{\prime}\not\in V(D[2])$ and $z^{\prime}\in V(D[0]\cup D[1])$.

$\mathbf {Case~3.1}$. $y^{\prime}=x$.

Without loss of generality, assume that $z^{\prime}\in V(D[0])$.
Recall that $(N_{D[0]}(x)\cup
N_{D[0]}(u)\cup N_{D[0]}(v))\backslash V(\mathcal{P})\neq\emptyset$.
If $N_{D[0]}(x)\backslash V(\mathcal{P})\neq\emptyset$,
then $|A_{1}|=r+s+2t=d_{\mathcal{P}}(x)< d_{D[0]}(x)=n+k-2$.
Choose $x_{1}$ in $D[0]$ such that $x_{1}=z^{\prime}$ and let $x_{1}\in A_{2}$.
Let $\hat{P}_{t+1}=yx\mathcal{F}[x,z^{\prime}]z^{\prime}z$.
Then $P_{1},\ldots,P_{r},\tilde{P}_{1},\ldots,\tilde{P}_{s},\hat{P}_{1},\ldots,\hat{P}_{t}$
and $\hat{P}_{t+1}$ are our desired ID$S$-paths in $D_{k,n}$.
\\
\indent
If $N_{D[0]}(u)\backslash V(\mathcal{P})\neq\emptyset$
(resp.~$N_{D[0]}(v)\backslash V(\mathcal{P})\neq\emptyset$),
similarly, $|B_{1}|< n+k-2$ (resp.~$|C_{1}|< n+k-2$).
We will delete the path $\hat{P}_{t}$ and find two paths $\hat{P}_{t+1}$ and $\hat{P}_{t+2}$.
Since the discussions are similar, we only consider $N_{D[0]}(u)\backslash V(\mathcal{P})\neq\emptyset$.
Choose $y_{1}$ in $D[1]$ such that $y_{1}^{\prime}\in V(D[2r+2s+2t+2])$ and let $y_{1}\in B_{2}$.
Let $\bar{U}_{3t}^{\prime}$ be the $(y_{1}^{\prime},z_{3t}^{\prime})$-path in $D[2r+2s+2t+2]$
as $D[2r+2s+2t+2]$ is connected and
let $A^{\prime}=A\backslash \{\bar{x}_{3t}\}\cup \{z^{\prime}\}$.
By Fan Lemma \ref{lem4},
there is a fan $\mathcal{F}^{\prime}$ from $x$ to $A^{\prime}$ which contains
$n+k-2$ internally disjoint $(x,A^{\prime})$-paths in $D[0]$.
Let $\hat{P}_{t+1}=\mathcal{F}^{\prime}[x,x_{3t}]x_{3t}x_{3t}^{\prime}U_{3t}[x_{3t}^{\prime},y_{3t}^{\prime}]
y_{3t}^{\prime}y_{3t}\mathcal{F}_{1}[y_{3t},y]y\mathcal{F}_{1}[y,y_{1}]y_{1}y_{1}^{\prime}
\bar{U}_{3t}^{\prime}[$
$y_{1}^{\prime},z_{3t}^{\prime}]z_{3t}^{\prime}z_{3t}\mathcal{F}_{2}[z_{3t},z]$
and $\hat{P}_{t+2}=yx\mathcal{F}^{\prime}[x,z^{\prime}]z^{\prime}z$ (see Figure~\ref{Fig.11}$~(a)$).
Then $P_{1}^{\prime},\ldots,P_{r}^{\prime},\tilde{P}_{1}^{\prime},
\ldots,\tilde{P}_{s}^{\prime},\hat{P}_{1}^{\prime},\ldots,\hat{P}_{t-1}^{\prime},\hat{P}_{t+1}$
and $\hat{P}_{t+2}$
are $\lfloor\frac{2n+3k}{4}\rfloor$ ID$S$-paths in $D_{k,n}$,
where $P_{1}^{\prime},\ldots,P_{r}^{\prime},\tilde{P}_{1}^{\prime},\ldots,
\tilde{P}_{s}^{\prime},\hat{P}_{1}^{\prime},$
$\ldots,\hat{P}_{t-1}^{\prime}$ are respectively obtained from
$P_{1},\ldots,P_{r},\tilde{P}_{1},\ldots,
\tilde{P}_{s},\hat{P}_{1},$
$\ldots,\hat{P}_{t-1}$ shown
in formulas $(1),(2)$ and $(3)$ by replacing the paths in the fan $\mathcal{F}$ with
the corresponding paths in the fan $\mathcal{F}^{\prime}$.

\begin{figure}[ht]
\begin{center}
\scalebox{0.5}[0.5]{\includegraphics{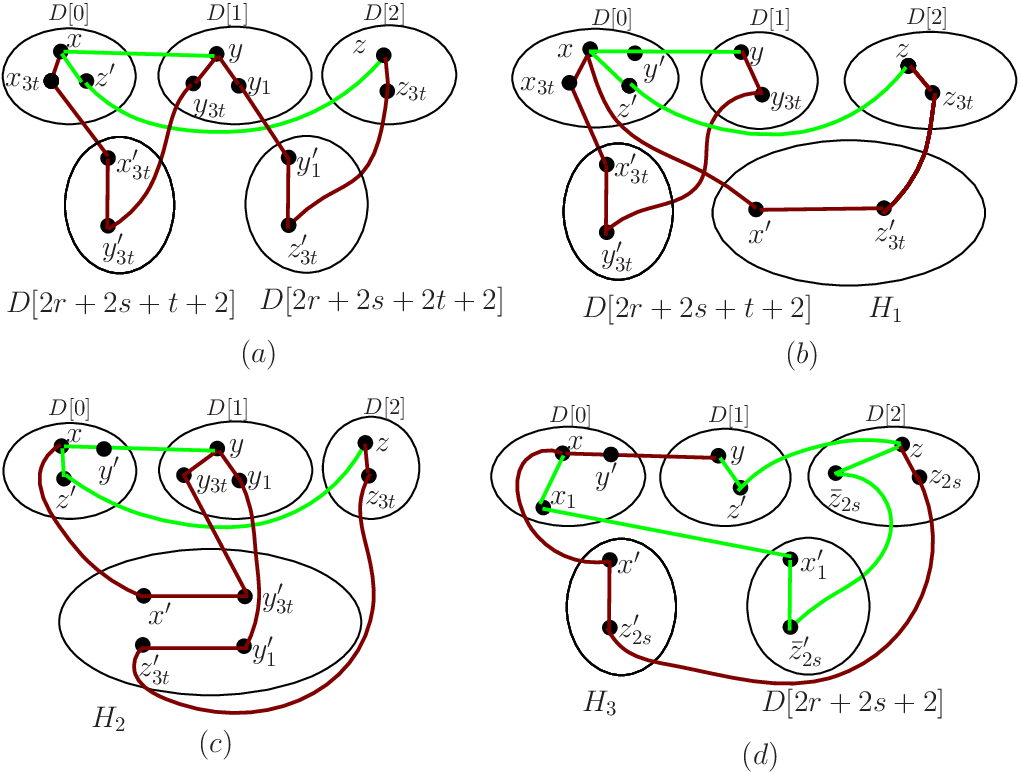}}\\
\captionsetup{font={small}}
\caption{The illustration of Case~$3$ in Claim~2}\label{Fig.11}
\end{center}
\end{figure}

$\mathbf {Case~3.2}$. $y^{\prime}\neq x$ and $z^{\prime}\in V(D[0])$.

Since $x^{\prime}\not\in V(D[2]\cup H_{0})$, $y^{\prime}\in V(D[0])$ and $y^{\prime}\neq x$,
one has that $x^{\prime}\in V(H)$ as Lemma \ref{lem1}~$(3)$.
Recall that $(N_{D[0]}(x)\cup
N_{D[0]}(u)\cup N_{D[0]}(v))\backslash V(\mathcal{P})\neq\emptyset$.
\\
\indent
If $N_{D[0]}(x)\backslash V(\mathcal{P})\neq\emptyset$,
then $|A_{1}|< n+k-2$.
Select $x_{1}$ in $D[0]$ such that $x_{1}=y^{\prime}$ and let $x_{1}\in A_{2}$.
Let $A^{\prime}=A\backslash \{\bar{x}_{3t}\}\cup \{z^{\prime}\}$,
by Fan Lemma \ref{lem4}, there is a fan $\mathcal{F}^{\prime}$
in $D[0]$ from $x$ to $A^{\prime}$.
Let $H_{1}=D_{k,n}[V(H)\cup V(D[2r+2s+2t+2])]$.
Then $H_{1}$ is connected.
Since $x^{\prime},z_{3t}^{\prime}\in V(H_{1})$,
there exists a path connecting $(x^{\prime},z_{3t}^{\prime})$, denoted by $L[x^{\prime},z_{3t}^{\prime}]$, in $H_{1}$.
Let $\hat{P}_{t+1}=\mathcal{F}_{1}[y,y_{3t}]y_{3t}y_{3t}^{\prime}U_{3t}[y_{3t}^{\prime},x_{3t}^{\prime}]
x_{3t}^{\prime}x_{3t}\mathcal{F}^{\prime}[x_{3t},x]xx^{\prime}L[x^{\prime},z_{3t}^{\prime}]z_{3t}^{\prime}z_{3t}
\mathcal{F}_{2}[z_{3t},z]$
and $\hat{P}_{t+2}=yy^{\prime}\mathcal{F}^{\prime}[y^{\prime},x]x\mathcal{F}^{\prime}[x,z^{\prime}]z^{\prime}z$
 (see Figure~\ref{Fig.11}$~(b)$).
Then $P_{1}^{\prime},\ldots,P_{r}^{\prime},\tilde{P}_{1}^{\prime},
\ldots,\tilde{P}_{s}^{\prime},\hat{P}_{1}^{\prime},\ldots,\hat{P}_{t-1}^{\prime},\hat{P}_{t+1}$
and $\hat{P}_{t+2}$
are $\lfloor\frac{2n+3k}{4}\rfloor$ ID$S$-paths in $D_{k,n}$,
where $P_{1}^{\prime},\ldots,P_{r}^{\prime},\tilde{P}_{1}^{\prime},
\ldots,\tilde{P}_{s}^{\prime},\hat{P}_{1}^{\prime},\ldots,\hat{P}_{t-1}^{\prime}$ are respectively
obtained from
$P_{1},\ldots,P_{r},\tilde{P}_{1},\ldots,\tilde{P}_{s},\hat{P}_{1},\ldots,\hat{P}_{t-1}$ shown
in formulas $(1),(2)$ and $(3)$ by replacing the paths in the fan $\mathcal{F}$ with
the corresponding paths in the fan $\mathcal{F}^{\prime}$.
\\
\indent
If $(N_{D[0]}(u)\cup N_{D[0]}(v))\backslash V(\mathcal{P})\neq\emptyset$,
without loss of generality, assume that $N_{D[0]}(u)\backslash V($
$\mathcal{P})\neq\emptyset$,
then $|B_{1}|< n+k-2$.
Choose $y_{1}$ in $D[1]$ such that $y_{1}^{\prime}\in V(H)$ and let $y_{1}\in B_{2}$.
Let $A^{\prime}=A\backslash\{\bar{x}_{3t},x_{3t}\}\cup \{z^{\prime},y^{\prime}\}$.
Then there is a fan, denoted by $\mathcal{F}^{\prime}$, in $D[0]$ from $x$ to $A^{\prime}$ as Fan Lemma \ref{lem4}.
Let $H_{2}=D_{k,n}[V(H)\cup V(D[2r+2s+t+2])\cup V(D[2r+2s+2t+2])],
D_{5}=\{a,b\}=\{x^{\prime},z_{3t}^{\prime}\}$ and $D_{6}=\{c,d\}=\{y_{3t}^{\prime},y_{1}^{\prime}\}$.
Notice that $D_{5},D_{6}\subseteq V(H_{2})$.
Then there are two internally disjoint $(D_{5},D_{6})$-paths,
say $\mathcal{D}_{5}[c,a]$ and $\mathcal{D}_{6}[b,d]$,
in $H_{2}$ as $H_{2}$ is $2$-connected.
Let $\hat{P}_{t+1}=yy^{\prime}\mathcal{F}^{\prime}[y^{\prime},x]x\mathcal{F}^{\prime}
[x,z^{\prime}]z^{\prime}z$
and $\hat{P}_{t+2}=xx^{\prime}\mathcal{D}_{6}[x^{\prime},y_{3t}^{\prime}]y_{3t}^{\prime}y_{3t}\mathcal{F}_{1}[y_{3t},y]
y\mathcal{F}_{1}[y,y_{1}]y_{1}y_{1}^{\prime}\mathcal{D}_{5}[y_{1}^{\prime},z_{3t}^{\prime}]
z_{3t}^{\prime}z_{3t}$ $\mathcal{F}_{2}[z_{3t},z]$ if $a=z_{3t}^{\prime}$ and $c=y_{1}^{\prime}$
(see Figure~\ref{Fig.11}$~(c)$).
Otherwise, $\hat{P}_{t+2}=xx^{\prime}\mathcal{D}_{6}[x^{\prime},y_{1}^{\prime}]y_{1}^{\prime}y_{1}$
$\mathcal{F}_{1}[y_{1},y]y\mathcal{F}_{1}[y,y_{3t}]y_{3t}y_{3t}^{\prime}\mathcal{D}_{5}$$[y_{3t}^{\prime},z_{3t}^{\prime}]
z_{3t}^{\prime}z_{3t}\mathcal{F}_{2}[z_{3t},z]$.
Then $P_{1}^{\prime},\ldots,P_{r}^{\prime},\tilde{P}_{1}^{\prime},
\ldots,\tilde{P}_{s}^{\prime},\hat{P}_{1}^{\prime},\ldots,\hat{P}_{t-1}^{\prime},$ $\hat{P}_{t+1}$
and $\hat{P}_{t+2}$
are $\lfloor\frac{2n+3k}{4}\rfloor$ ID$S$-paths in $D_{k,n}$,
where $P_{1}^{\prime},\ldots,P_{r}^{\prime},\tilde{P}_{1}^{\prime},
\ldots,\tilde{P}_{s}^{\prime},\hat{P}_{1}^{\prime},\ldots,\hat{P}_{t-1}^{\prime}$ are respectively
obtained from
$P_{1},\ldots,P_{r},\tilde{P}_{1},\ldots,\tilde{P}_{s},\hat{P}_{1},\ldots,\hat{P}_{t-1}$ shown
in formulas $(1),(2)$ and $(3)$ by replacing the paths in the fan $\mathcal{F}$ with
the corresponding paths in the fan $\mathcal{F}^{\prime}$.

$\mathbf {Case~3.3}$. $y^{\prime}\neq x$ and $z^{\prime}\in V(D[1])$.

Similar to Case~$3.2$,
one has that $x^{\prime}\in V(H)$.
Recall that $|(N_{D[0]}(x)\cup N_{D[0]}(u)\cup N_{D[0]}(v))\backslash V(\mathcal{P})|\geq 2$.
\\
\indent
If $|N_{D[0]}(x)\backslash V(\mathcal{P})|\geq 2$
(resp.~$|N_{D[0]}(u)\backslash V(\mathcal{P})|\geq 2,$ or
$|N_{D[0]}(v)\backslash V(\mathcal{P})|\geq 2$),
then $|A_{1}|\leq n+k-4$
(resp.~$|B_{1}|\leq n+k-4$, or $|C_{1}|\leq n+k-4$).
We will delete the path $\tilde{P}_{s}$ (resp.~$\hat{P}_{t},$ or $P_{r}$)
and find two paths $\tilde{P}_{s+1}$ and $\tilde{P}_{s+2}$
(resp.~$\hat{P}_{t+1}$ and $\hat{P}_{t+2}$, or $P_{r+1}$ and $P_{r+2}$).
Since the discussions are similar, we only consider
$|N_{D[0]}(x)\backslash V(\mathcal{P})|\geq 2$.
Choose $x_{1}$ and $x_{2}$ in $D[0]$
such that $x_{1}^{\prime}\in V(D[2r+2s+2]),x_{2}^{\prime}=y^{\prime}$ and let $x_{1},x_{2}\in A_{2}$.
As $D[2r+2s+2]$ is connected, there is a $(x_{1}^{\prime},\bar{z}_{2s}^{\prime})$-path,
denoted by $\bar{U}_{2s}^{\prime}$, in $D[2r+2s+2]$.
Let $H_{3}=D_{k,n}[V(H)\cup V(D[2r+s+2])]$.
Then there is a $(x^{\prime},z_{2s}^{\prime})$-path, denoted by $L[x^{\prime},z_{2s}^{\prime}]$, in $H_{3}$.
Let $\tilde{P}_{s+1}=yy^{\prime}\mathcal{F}[y^{\prime},x]xx^{\prime}L[x^{\prime},z_{2s}^{\prime}]
z_{2s}^{\prime}z_{2s}\mathcal{F}_{2}[z_{2s},z]$ and
$\tilde{P}_{s+2}=\mathcal{F}[x,x_{1}]x_{1}x_{1}^{\prime}\bar{U}_{2s}^{\prime}
[x_{1}^{\prime},\bar{z}_{2s}^{\prime}]\bar{z}_{2s}^{\prime}\bar{z}_{2s}\mathcal{F}_{2}[\bar{z}_{2s},z]
zz^{\prime}$ $\mathcal{F}_{1}[z^{\prime},y]$ (see Figure~\ref{Fig.11}$~(d)$).
Then $P_{1},\ldots,$ $P_{r},\tilde{P}_{1},\ldots,\tilde{P}_{s-1},\tilde{P}_{s+1},\tilde{P}_{s+2},
\hat{P}_{1},\ldots,\hat{P}_{t}$
are $\lfloor\frac{2n+3k}{4}\rfloor$ ID$S$-paths in $D_{k,n}$.
\\
\indent
If each of $\{|N_{D[0]}(x)\backslash V(\mathcal{P})|,|N_{D[0]}(u)\backslash V(\mathcal{P})|,
|N_{D[0]}(v)\backslash V(\mathcal{P})|\}$ is at most $1$,
as $|N_{D[0]}(x)\cup N_{D[0]}(u)\cup N_{D[0]}(v)\backslash V(\mathcal{P})|\geq 2$,
then at least two of them are $1$.
\\
\indent
If $|N_{D[0]}(x)\backslash V(\mathcal{P})|=1$ and $|N_{D[0]}(u)\backslash V(\mathcal{P})|=1$,
then $|A_{1}|\leq n+k-3$ and $|B_{1}|\leq n+k-3$.
Choose $x_{1}$ in $D[0]$ and $y_{1}$ in $D[1]$ such that $x_{1}=y^{\prime},y_{1}=z^{\prime}$
and let $x_{1}\in A_{2}$ and $y_{1}\in B_{2}$.
Let $\hat{P}_{t+1}=\mathcal{F}[x,y^{\prime}]y^{\prime}y\mathcal{F}_{1}[y,z^{\prime}]z^{\prime}z$.
Then $P_{1},\ldots,P_{r},\tilde{P}_{1},\ldots,\tilde{P}_{s},\hat{P}_{1},\ldots,\hat{P}_{t}$
and $\hat{P}_{t+1}$
are $\lfloor\frac{2n+3k}{4}\rfloor$ ID$S$-paths in $D_{k,n}$.
\\
\indent
If $|N_{D[0]}(x)\backslash V(\mathcal{P})|=1$ and $|N_{D[0]}(v)\backslash V(\mathcal{P})|=1$,
or $|N_{D[0]}(u)\backslash V(\mathcal{P})|=1$ and $|N_{D[0]}(v)\backslash$ $V(\mathcal{P})|=1$,
without loss of generality, assume that
$|N_{D[0]}(x)\backslash V(\mathcal{P})|=1$ and $|N_{D[0]}(v)\backslash V(\mathcal{P})|=1$.
Then $|A_{1}|\leq n+k-3$ and $|C_{1}|\leq n+k-3$.
Choose $x_{1}$ in $D[0]$ and $z_{1}$ in $D[2]$ such that $x_{1}=y^{\prime},z_{1}^{\prime}\in V(H)$
and let $x_{1}\in A_{2}$ and $z_{1}\in C_{2}$.
Then there is a $(x^{\prime},z_{1}^{\prime})$-path, say $L[x^{\prime},z_{1}^{\prime}]$,
in $H$ as $H$ is connected.
Let $\hat{P}_{t+1}=yy^{\prime}\mathcal{F}[y^{\prime},x]
xx^{\prime}L[x^{\prime},z_{1}^{\prime}]z_{1}^{\prime}z_{1}
\mathcal{F}_{2}[z_{1},z]$.
Then $P_{1},\ldots,P_{r},\tilde{P}_{1},\ldots,\tilde{P}_{s},\hat{P}_{1},\ldots,\hat{P}_{t}$
and $\hat{P}_{t+1}$
are $\lfloor\frac{2n+3k}{4}\rfloor$ ID$S$-paths in $D_{k,n}$.

$\mathbf {Case~4}$. $y^{\prime}\in V(D[0])$ and $x^{\prime}\in V(D[2])$.

If $z^{\prime}\not\in V(D[0]\cup D[1])$, we have a similar argument to Case~$3.3$.
If $z^{\prime}\in V(D[0])$, then $x^{\prime}=z$, which is similar to Case~$3.1$.
Next, we consider $z^{\prime}\in V(D[1])$.
\\
\indent
Recall that
$|(N_{D[0]}(x)\cup
N_{D[0]}(u)\cup N_{D[0]}(v))\backslash V(\mathcal{P})|\geq 2$.
If at least one of $\{|N_{D[0]}(x)\backslash V(\mathcal{P})|,$ $|N_{D[0]}(u)\backslash V(\mathcal{P})|,
|N_{D[0]}(v)\backslash V(\mathcal{P})|\}$ is at least $2$,
without loss of generality, assume that
$|N_{D[0]}(x)$ $\backslash V(\mathcal{P})|\geq 2$.
Then $|A_{1}|\leq n+k-4$.
Choose $x_{1}$ in $D[0]$
such that $x_{1}=y^{\prime}$ and let $x_{1}\in A_{2}$.
Let $B^{\prime}=B\backslash\{y_{2s}\}\cup \{z^{\prime}\}$
and $C^{\prime}=C\backslash\{\bar{z}_{2s}\}\cup \{x^{\prime}\}$.
By Fan Lemma \ref{lem4},
there are two fans, say $\mathcal{F}_{1}^{\prime}$ and $\mathcal{F}_{2}^{\prime}$,
in $D[1]$ and $D[2]$, respectively.
Let $\tilde{P}_{s+1}=yy^{\prime}\mathcal{F}[y^{\prime},x]x\mathcal{F}[x,x_{2s}]x_{2s}x_{2s}^{\prime}
U_{2s}[x_{2s}^{\prime},z_{2s}^{\prime}]z_{2s}^{\prime}$ $z_{2s}\mathcal{F}_{2}^{\prime}[z_{2s},z]$
and $\tilde{P}_{s+2}=xx^{\prime}\mathcal{F}_{2}^{\prime}[x^{\prime},z]zz^{\prime}
\mathcal{F}_{1}^{\prime}[z^{\prime},y]$.
Then $P_{1}^{\prime},\ldots,P_{r}^{\prime},\tilde{P}_{1}^{\prime},\ldots,\tilde{P}_{s-1}^{\prime},
\tilde{P}_{s+1},\tilde{P}_{s+2},\hat{P}_{1}^{\prime}$ $,\ldots,\hat{P}_{t}^{\prime}$
are $\lfloor\frac{2n+3k}{4}\rfloor$ ID$S$-paths in $D_{k,n}$,
where $P_{1}^{\prime},\ldots,P_{r}^{\prime},\tilde{P}_{1}^{\prime},
\ldots,\tilde{P}_{s-1}^{\prime},\hat{P}_{1}^{\prime},\ldots,\hat{P}_{t}^{\prime}$ are respectively
obtained from
$P_{1},\ldots,P_{r},\tilde{P}_{1},\ldots,\tilde{P}_{s-1},\hat{P}_{1},\ldots,\hat{P}_{t}$ shown
in formulas $(1),(2)$ and $(3)$ by replacing the paths in the fans $\mathcal{F}_{1}$ and $\mathcal{F}_{2}$ with
the corresponding paths in the fans $\mathcal{F}_{1}^{\prime}$ and $\mathcal{F}_{2}^{\prime}$, respectively.
\\
\indent
If each of $\{|N_{D[0]}(x)\backslash V(\mathcal{P})|,$ $|N_{D[0]}(u)\backslash V(\mathcal{P})|,
|N_{D[0]}(v)\backslash V(\mathcal{P})|\}$ is at most $1$,
by $|(N_{D[0]}(x)\cup
N_{D[0]}(u)\cup N_{D[0]}(v))\backslash V(\mathcal{P})|\geq 2$,
one has that at least two of them are $1$.
Without loss of generality, assume that
$|N_{D[0]}(x)\backslash V(\mathcal{P})|=1$ and $|N_{D[0]}(u)\backslash V(\mathcal{P})|=1$,
then $|A_{1}|\leq n+k-3$ and $|B_{1}|\leq n+k-3$.
Choose $x_{1}$ in $D[0]$ and $y_{1}$ in $D[1]$ such that $x_{1}=y^{\prime},y_{1}=z^{\prime}$
and let $x_{1}\in A_{2}$ and $y_{1}\in B_{2}$.
Let $\hat{P}_{t+1}=\mathcal{F}[x,y^{\prime}]y^{\prime}y\mathcal{F}_{1}[y,z^{\prime}]
z^{\prime}z$.
Then $P_{1},\ldots,P_{r},\tilde{P}_{1},\ldots,\tilde{P}_{s},\hat{P}_{1},\ldots,\hat{P}_{t}$
and $\hat{P}_{t+1}$
are $\lfloor\frac{2n+3k}{4}\rfloor$ ID$S$-paths in $D_{k,n}$.
$\hfill\blacksquare$

\begin{Claim}
There exist at least $ \lfloor \frac{2n+3k}{4}\rfloor$ ID$S$-paths in $D_{k,n}$,
if the vertices of $S$ belong to two distinct copies of $D_{k-1,n}$.
\end{Claim}

\noindent {\it Proof of Claim 3.}
Without loss of generality, assume that $x,y\in V(D[0])$ and $z\in V(D[1])$.
Choose $u\in V(D[0])$ such that $u\not\in N_{D[0]}[x]\cup N_{D[0]}[y]$ and
at least two vertices of $\{x,y,u\}$ belong to two distinct copies of $D_{k-2,n}$.
This can be done as $D[0]\cong D_{k-1,n}$ for $k\geq 2$ is not the completed graph
and $D[0]$ is the union of the copies of $D_{k-2,n}$.
Let $S_{0}=\{x,y,u\}$.
Define $\xi=|N_{D[0]}(x)\cap N_{D[0]}(y)\cap N_{D[0]}(u)|$.
Then $\xi \leq 1$ as Lemma \ref{lem1}~($2$) and Lemma \ref{lem1}~($4$) and the chosen of $u$.
By the inductive hypothesis,
$\pi_{3}(D[0])=\lfloor\frac{2n+3(k-1)}{4}\rfloor$.
Then there is a set $\mathcal{T}$ of $\pi_{3}(D[0])$
ID$S_{0}$-paths in $D[0]$.
When $n\equiv 0($mod~$2)$ and $k\not\equiv 1($mod~$4)$,
or $n\equiv 1($mod~$2)$ and $k\not\equiv 3($mod~$4)$,
by calculation, $d_{D[0]}(x)+d_{D[0]}(y)+d_{D[0]}(u)-4\pi_{3}(D[0])-\xi=
3(n+k-2)-4\times \lfloor\frac{2n+3(k-1)}{4}\rfloor-1\geq 3$.
By Lemma \ref{thm21}~(2),
then there exist at least two vertices, say $v$ and $v_{1}$,
and a set $\mathcal{T}_{0}$ of $\lfloor\frac{2n+3(k-1)}{4}\rfloor$ ID$S_{0}$-paths in $D[0]$
such that $v,v_{1}\in (N_{D[0]}(x)\cup N_{D[0]}(y)\cup N_{D[0]}(u))\backslash V(\mathcal{T}_{0})$.
Without loss of generality, assume that $\mathcal{T}=\mathcal{T}_{0}$ and
$\mathcal{T}=\{T_{1},T_{2},\ldots, T_{h}|h=\lfloor\frac{2n+3(k-1)}{4}\rfloor\}$.
Then $v,v_{1}\in (N_{D[0]}(x)\cup N_{D[0]}(y)\cup N_{D[0]}(u))\backslash V(\mathcal{T})$.
Suppose that $u$ is one of the terminal vertices of $T_{i}$ with $uu_{1i}\in E(T_{i})$ for each $1\leq i\leq r$,
and $u$ is the internal vertex of $T_{r+j}$ with $uu_{2j},u\bar{u}_{2j}\in E(T_{r+j})$ for each $1\leq j\leq s$,
where $r+s=h=\lfloor\frac{2n+3(k-1)}{4}\rfloor$.
Moreover, let $x$ be also the terminal vertex of $T_{r}$.
Without loss of generality, assume that $d_{T_{r+j}}(x,u_{2j})< d_{T_{r+j}}(x,\bar{u}_{2j})$ for each $1\leq j\leq s$.
Since the proofs for $\{u_{1i}^{\prime},u_{2j}^{\prime},\bar{u}_{2j}^{\prime}|1\leq i\leq r,1\leq j\leq s,
r+s=\lfloor\frac{2n+3(k-1)}{4}\rfloor\}\cap V(D[1])\neq\emptyset$ and
$\{u_{1i}^{\prime},u_{2j}^{\prime},\bar{u}_{2j}^{\prime}|1\leq i\leq r,1\leq j\leq s,
r+s=\lfloor\frac{2n+3(k-1)}{4}\rfloor\}\cap V(D[1])=\emptyset$ are similar,
we only prove the latter.
By Lemma \ref{lem1}~$(3)$,
without loss of generality, assume that $u_{1i}^{\prime}\in V(D[i+1]),
u_{2j}^{\prime}\in V(D[r+j+1])$ and $\bar{u}_{2j}^{\prime}\in V(D[r+s+j+1])$
for $1\leq i\leq r$ and $1\leq j\leq s$.
And
let $z_{1i},z_{2j}$ and $\bar{z}_{2j}$ for $1\leq i\leq r$ and $1\leq j\leq s$ be vertices in $D[1]$
such that $z_{1i}^{\prime}\in V(D[i+1]),
z_{2j}^{\prime}\in V(D[r+j+1])$ and $\bar{z}_{2j}^{\prime}\in V(D[r+s+j+1])$.
Denoted by
\begin{center}
$A=\{z_{1i},z_{2j},\bar{z}_{2j}|1\leq i\leq r,1\leq j\leq s,
r+s=\lfloor\frac{2n+3(k-1)}{4}\rfloor\}$ (May be $z\in A$).
\end{center}

\indent
As $r+2s\leq d_{D[0]}(u)\leq n+k-2$,
by Fan Lemma \ref{lem4}, there exists a fan $\mathcal{O}$ which contains
$r+2s$ internally disjoint paths from $z$ to $A$ in $D[1]$
(If $z\in A$,
without loss of generality, assume that $z=z_{1r}$,
then the path $\mathcal{O}[z,z_{1r}]$ is replaced with $z$).
Since $D[t]$ is connected with $t\in \langle t_{k-1,n}+1\rangle$,
there exists one path connecting any two distinct vertices in $D[t]$.
Suppose that $R_{1i},R_{2j}$ and $\bar{R}_{2j}$ are the $(u_{1i}^{\prime},z_{1i}^{\prime})$-path,
$(u_{2j}^{\prime},z_{2j}^{\prime})$-path and $(\bar{u}_{2j}^{\prime},\bar{z}_{2j}^{\prime})$-path
in $D[i+1],D[r+j+1]$ and $D[r+s+j+1]$, respectively,
for $1\leq i\leq r$ and $1\leq j\leq s$.
Let
$$
\begin{aligned}
P_{i}=T_{i}[a,u_{1i}]u_{1i}u_{1i}^{\prime}R_{1i}[u_{1i}^{\prime},z_{1i}^{\prime}]z_{1i}^{\prime}z_{1i}
\mathcal{O}[z_{1i},z]~~~~~~~~~~~~~
~~~~~~~~~~~~~~~~~~~~~~~~~~~~~&(4),\\
\tilde{P}_{j}=T_{r+j}[x,u_{2j}]u_{2j}u_{2j}^{\prime}R_{2j}[u_{2j}^{\prime},
z_{2j}^{\prime}]z_{2j}^{\prime}z_{2j}
\mathcal{O}[z_{2j},z]z\mathcal{O}[z,\bar{z}_{2j}]\bar{z}_{2j}\bar{z}_{2j}^{\prime}
\bar{R}_{2j}[\bar{z}_{2j}^{\prime},\bar{u}_{2j}^{\prime}]\bar{u}_{2j}^{\prime}\bar{u}_{2j}
T_{r+j}[\bar{u}_{2j},y]~&(5)
\end{aligned}
$$
with $1\leq i\leq r,1\leq j\leq s$ and $r+s=\lfloor\frac{2n+3(k-1)}{4}\rfloor$,
where $a=x$ if $d_{T_{i}}(x)=1$, otherwise, $a=y$,
and for $\mu,\nu\in T_{\varepsilon}$, let
$T_{\varepsilon}[\mu,\nu]$ be the subpath of $T_{\varepsilon}$ which connects $\mu$ and $\nu$ for $\varepsilon\in [r+s]$.

$\mathbf {Part~I}$. $n\equiv 0($mod~$2)$ and $k\equiv 1($mod~$4)$,
or $n\equiv 1($mod~$2)$ and $k\equiv 3($mod~$4)$.

If $n\equiv 0($mod~$2)$ and $k\equiv 1($mod~$4)$,
or $n\equiv 1($mod~$2)$ and $k\equiv 3($mod~$4)$,
then $\lfloor\frac{2n+3(k-1)}{4}\rfloor=\lfloor\frac{2n+3(k)}{4}\rfloor$.
Hence $P_{1},\ldots,P_{r},\tilde{P}_{1},\ldots,\tilde{P}_{s}$ are our desired
ID$S$-paths in $D_{k,n}$, which are represented by distinct colors
(see Figure~\ref{Fig.5}~$(a)$).

$\mathbf {Part~II}$. $n\equiv 0($mod~$2)$ and $k\not\equiv 1($mod~$4)$,
or $n\equiv 1($mod~$2)$ and $k\not\equiv 3($mod~$4)$.

If $n\equiv 0($mod~$2)$ and $k\not\equiv 1($mod~$4)$,
or $n\equiv 1($mod~$2)$ and $k\not\equiv 3($mod~$4)$,
then $\lfloor\frac{2n+3(k-1)}{4}\rfloor+1=\lfloor\frac{2n+3k}{4}\rfloor$.
Thus we need to find one more desired $S$-path in $D_{k,n}$.
Recall that $v,v_{1}\in (N_{D[0]}(x)\cup N_{D[0]}(y)\cup N_{D[0]}(u))\backslash V(\mathcal{T})$.
Let
$\bm{H=D_{k,n}[\bigcup _{\ell=r+2s+2}^{t_{k-1,n}}V(D[\ell])]}~\bf{for}~\bm{k\geq 2}.$
To prove the result, the following cases are considered.

$\mathbf {Case~1}$. $\{x^{\prime},y^{\prime}\}\cap V(D[1])\neq \emptyset$,
where $\alpha^{\prime}$ is the $k$-dimensional neighbour of the vertex $\alpha$
for $\alpha\in\{x,y\}$ in $D_{k,n}$.

Without loss of generality, assume that $x^{\prime}\in V(D[1])$ and
$y^{\prime}\in V(H)$ as Lemma \ref{lem1}~$(3)$.

\begin{figure}[ht]
\begin{center}
\scalebox{0.5}[0.5]{\includegraphics{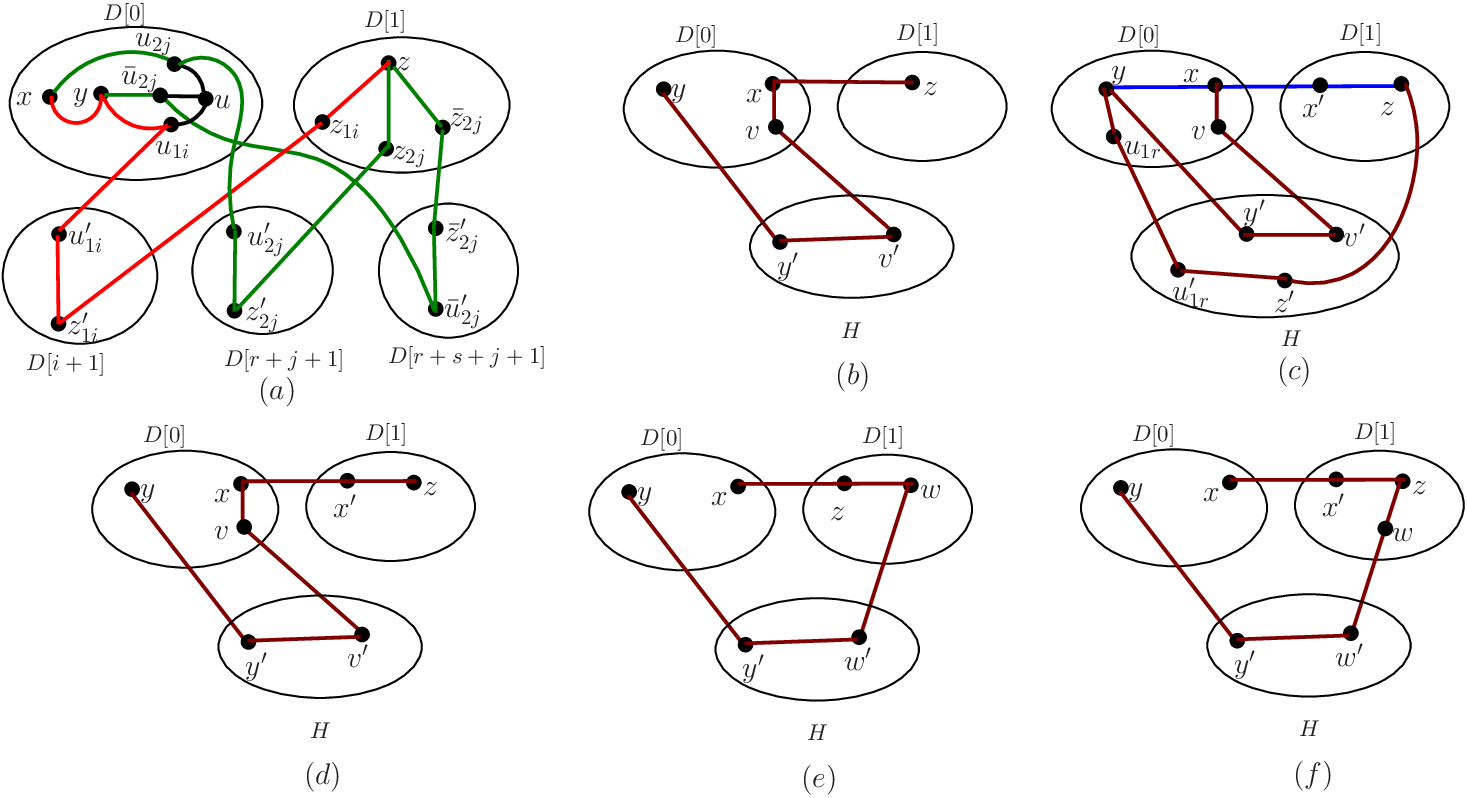}}\\
\captionsetup{font={small}}
\caption{The illustration of Case $1$ in Claim~3}\label{Fig.5}
\end{center}
\end{figure}

$\mathbf {Case~1.1}$. $d_{\mathcal{T}}(u)=d_{D[0]}(u)$.

Notice that there is a vertex $v\in (N_{D[0]}(x)\cup N_{D[0]}(y)\cup N_{D[0]}(u))\backslash V(\mathcal{T})$.
As $d_{\mathcal{T}}(u)=d_{D[0]}(u)$,
$v\in (N_{D[0]}(x)\cup N_{D[0]}(y))\backslash V(\mathcal{T})$.
Without loss of generality, assume that $v\in N_{D[0]}(x)\backslash$ $V(\mathcal{T})$.
By Lemma \ref{lem1}~$(3)$, $v^{\prime}\in V(H)$.
By Lemma \ref{lem5},
there exists a $(y^{\prime},v^{\prime})$-path, say $U_{0}$, in $H$.
\\
\indent
If $x^{\prime}=z$,
let $\tilde{P}_{s+1}=yy^{\prime}U_{0}[y^{\prime},v^{\prime}]v^{\prime}vxz$ (see Figure~\ref{Fig.5}$~(b)$).
Then $P_{1},\ldots,P_{r},\tilde{P}_{1},\ldots,\tilde{P}_{s}$ and $\tilde{P}_{s+1}$
are our desired $\lfloor\frac{2n+3k}{4}\rfloor$ ID$S$-paths in $D_{k,n}$.
\\
\indent
If $x^{\prime}\neq z$,
let $A^{\prime}=A\backslash \{\eta\}\cup \{x^{\prime}\}$.
If $z\not\in A$,
then $z^{\prime}\in V(H)$,
without loss of generality, assume that $\eta=z_{1r}$.
If $z\in A$, without loss of generality, assume that $z=z_{1r}$ and let $\eta=z_{1r}$,
then regard $\mathcal{O}[z_{1r},z]$ as the vertex $z$ in $P_{r}$ which is shown in formula $(4)$.
Thus
$A^{\prime}=\{z_{1i_{1}},x^{\prime},z_{2j},\bar{z}_{2j}|1\leq i_{1}\leq r-1,1\leq j\leq s,
r+s=\lfloor\frac{2n+3(k-1)}{4}\rfloor\}$ whether $z\in A$ or not.
By Fan Lemma \ref{lem4},
there is a fan $\mathcal{O}^{\prime}$ in $D[1]$ from $z$ to $A^{\prime}$.
Let $P_{i_{1}}^{\prime}$ be the path obtained from
$P_{i_{1}}$ shown in formula $(4)$
by replacing
$\mathcal{O}[z_{1i_{1}},z]$ with
$\mathcal{O}^{\prime}[z_{1i_{1}},z]$ for $1\leq i_{1}\leq r-1$.
Let $\tilde{P}_{j}^{\prime}$ denote the path from $\tilde{P}_{j}$ shown in formula $(5)$ by replacing
$\mathcal{O}[z_{2j},z]$ and $\mathcal{O}[z,\bar{z}_{2j}]$ with $\mathcal{O}^{\prime}[z_{2j},z]$
and $\mathcal{O}^{\prime}[z,\bar{z}_{2j}]$, respectively, where $1\leq j\leq s$.
When $z\not\in A$, we will destroy the path $P_{r}$ and find two other paths $P_{r+1}$ and $P_{r+2}$.
Let $H_{1}=D_{k,n}[V(H)\cup V(D[r+1])]$.
Then $H_{1}$ is $2$-connected as Lemma \ref{lem5}.
Let $B_{1}=\{a,b\}=\{y^{\prime},u_{1r}^{\prime}\}$ and $B_{2}=\{c,d\}=\{v^{\prime},z^{\prime}\}$.
Then there are two disjoint paths, denoted by $\mathcal{B}_{1}[c,a]$ and $\mathcal{B}_{2}[b,d]$,
connecting $(B_{1},B_{2})$ in $H_{1}$.
Let
$P_{r+1}=T_{r}[y,x]xx^{\prime}\mathcal{O}^{\prime}[x^{\prime},z]$
and
$P_{r+2}=xvv^{\prime}\mathcal{B}_{1}[v^{\prime},y^{\prime}]y^{\prime}y
T_{r}[y,u_{1r}]u_{1r}u_{1r}^{\prime}\mathcal{B}_{2}[u_{1r}^{\prime},z^{\prime}]z^{\prime}z$
if $a=y^{\prime}$ and $c=v^{\prime}$ (see Figure~\ref{Fig.5}$~(c)$).
Otherwise, $P_{r+2}=xvv^{\prime}\mathcal{B}_{1}[v^{\prime},u_{1r}^{\prime}]u_{1r}^{\prime}u_{1r}
T_{r}[u_{1r},y]yy^{\prime}\mathcal{B}_{2}[y^{\prime},z^{\prime}]z^{\prime}z$.
Then $P_{1}^{\prime},\ldots,P_{r-1}^{\prime},P_{r+1},P_{r+2},\tilde{P}_{1}^{\prime},$
$\ldots, \tilde{P}_{s}^{\prime}$
are our desired ID$S$-paths in $D_{k,n}$.
When $z\in A$,
let
$\tilde{P}_{s+1}^{\prime}=\mathcal{O}^{\prime}[z,x^{\prime}]x^{\prime}xvv^{\prime}U_{0}[v^{\prime},y^{\prime}]y^{\prime}y$
(see Figure~\ref{Fig.5}$~(d)$).
Then $P_{1}^{\prime},\ldots,P_{r-1}^{\prime},P_{r},\tilde{P}_{1}^{\prime},\ldots,\tilde{P}_{s}^{\prime}$
and $\tilde{P}_{s+1}^{\prime}$
are $\lfloor\frac{2n+3k}{4}\rfloor$ ID$S$-paths in $D_{k,n}$.

\indent
$\mathbf {Case~1.2}$. $d_{\mathcal{T}}(u)<d_{D[0]}(u)$.

\indent
As $|A|=r+2s=d_{\mathcal{T}}(u)<d_{D[0]}(u)=n+k-2$,
let $A^{\prime}=A\cup \{w\}$,
where $w\in V(D[1])\backslash (A\cup \{z\})$.
Then there is a fan $\mathcal{O}^{\prime}$ from $z$ to $A^{\prime}$
which contains $r+2s+1$ internally disjoint $(z,A^{\prime})$-paths in $D[1]$.
Let $P_{1}^{\prime},\ldots,P_{r}^{\prime},\tilde{P}_{1}^{\prime},\ldots,\tilde{P}_{s}^{\prime}$
be respectively obtained from
$P_{1},\ldots,P_{r},\tilde{P}_{1},\ldots,\tilde{P}_{s}$ shown in formulas $(4)$ and $(5)$ by
replacing the paths in the fan $\mathcal{O}$ with the
corresponding paths in the fan $\mathcal{O}^{\prime}$.
\\
\indent
If $x^{\prime}\neq z$ and $z\not\in A$,
it implies that $z^{\prime}\in V(H)$.
Choose $w$ such that $w=x^{\prime}$
and let $\tilde{U}_{0}^{\prime}$ be the $(y^{\prime},z^{\prime})$-path in $H$.
Let $\tilde{P}_{s+1}=xx^{\prime}\mathcal{O}^{\prime}[x^{\prime},z]zz^{\prime}\tilde{U}_{0}^{\prime}
[z^{\prime},y^{\prime}]y^{\prime}y$.
Then $P_{1}^{\prime},\ldots,P_{r}^{\prime},\tilde{P}_{1}^{\prime},\ldots,\tilde{P}_{s}^{\prime}$ and $\tilde{P}_{s+1}$
are $\lfloor\frac{2n+3k}{4}\rfloor$ ID$S$-paths in $D_{k,n}$.
\\
\indent
If $x^{\prime}=z$ or $x^{\prime}\neq z$ and $z\in A$,
choose $w$ such that $w^{\prime}\in V(H)$
and $\tilde{U}_{0}$ be the $(y^{\prime},w^{\prime})$-path in $H$.
If $x^{\prime}=z$,
let $\tilde{P}_{s+1}=xz\mathcal{O}^{\prime}[z,w]ww^{\prime}\tilde{U}_{0}[w^{\prime},y^{\prime}]y^{\prime}y$
 (see Figure~\ref{Fig.5}$~(e)$).
Then $P_{1}^{\prime},\ldots,P_{r}^{\prime},\tilde{P}_{1}^{\prime},\ldots,\tilde{P}_{s}^{\prime}$
and $\tilde{P}_{s+1}$
are $\lfloor\frac{2n+3k}{4}\rfloor$ ID$S$-paths in $D_{k,n}$.
If $x^{\prime}\neq z$ and $z\in A$,
without loss of generality, let $z=z_{1r}$,
then regard the path $\mathcal{O}[z,z_{1r}]$ as the vertex $z$ in $P_{r}$.
Let $A^{\prime\prime}=A\backslash\{z_{1r}\}\cup \{w,x^{\prime}\}$.
Then there exists a fan $\mathcal{O}^{\prime\prime}$ which contains
$r+2s+1$ internally disjoint paths from $z$ to $A^{\prime\prime}$ in $D[1]$.
Let $\tilde{P}_{s+1}^{\prime}=xx^{\prime}\mathcal{O}^{\prime\prime}[x^{\prime},z]z\mathcal{O}^{\prime\prime}[z,w]
ww^{\prime}\tilde{U}_{0}[w^{\prime},y^{\prime}]y^{\prime}y$ (see Figure~\ref{Fig.5}$~(f)$).
Then $P_{1}^{\prime\prime},\ldots,P_{r-1}^{\prime\prime},P_{r},\tilde{P}_{1}^{\prime\prime},
\ldots,\tilde{P}_{s}^{\prime\prime}$ and $\tilde{P}_{s+1}^{\prime}$
are $\lfloor\frac{2n+3k}{4}\rfloor$ ID$S$-paths in $D_{k,n}$,
where $P_{1}^{\prime\prime},\ldots,P_{r-1}^{\prime\prime},\tilde{P}_{1}^{\prime\prime},
\ldots,\tilde{P}_{s}^{\prime\prime}$ are respectively obtained from
$P_{1},\ldots,P_{r-1},\tilde{P}_{1},\ldots,\tilde{P}_{s}$ in formula $(4)$ and $(5)$
just by replacing the paths in the fan $\mathcal{O}$ with the
 corresponding paths in the fan $\mathcal{O}^{\prime\prime}$.

\indent
$\mathbf {Case~2}$. $\{x^{\prime},y^{\prime}\}\cap V(D[1])=\emptyset$.

\indent
As $\{x^{\prime},y^{\prime}\}\cap V(D[1])=\emptyset$ and Lemma \ref{lem1}~$(3)$,
one has that $\{x^{\prime},y^{\prime}\}\subseteq V(H)$,
where $H=D_{k,n}[\bigcup _{\ell=r+2s+2}^{t_{k-1,n}}V(D[\ell])]$.
Suppose that $\hat{U}_{0}$ is the $(y^{\prime},x^{\prime})$-path in $H$ as $H$ is connected.
Let each one in $\{B_{1},B_{2},\mathcal{B}_{1}[c,a],\mathcal{B}_{2}[b,d]\}$ be the same as that in Case~$1.1$,
where $B_{1}=\{a,b\}=\{y^{\prime},u_{1r}^{\prime}\}$ and $B_{2}=\{c,d\}=\{v^{\prime},z^{\prime}\}$.

\indent
$\mathbf {Case~2.1}$. $d_{\mathcal{T}}(u)=d_{D[0]}(u)$.

\indent
Recall that there is a vertex $v\in (N_{D[0]}(x)\cup N_{D[0]}(y)\cup N_{D[0]}(u))\backslash V(\mathcal{T})$.
Since $d_{\mathcal{T}}(u)=d_{D[0]}(u)$, $v\in (N_{D[0]}(x)\cup N_{D[0]}(y))\backslash V(\mathcal{T})$.
Without loss of generality, assume that
$v\in N_{D[0]}(x)\backslash V$ $(\mathcal{T})$.
By Lemma \ref{lem1}~$(3)$,
one has that either $v^{\prime}\in V(D[1])$ or $v^{\prime}\in V(H)$.

\indent
$\mathbf {Case~2.1.1}$. $v^{\prime}\in V(D[1])$.

\indent
If $v^{\prime}=z$,
let $\tilde{P}_{s+1}=zvxx^{\prime}\hat{U}_{0}[x^{\prime},y^{\prime}]y^{\prime}y$.
Then $P_{1},\ldots,P_{r},\tilde{P}_{1},\ldots,\tilde{P}_{s}$ and $\tilde{P}_{s+1}$
are our desired ID$S$-paths in $D_{k,n}$.
\\
\indent
If $v^{\prime}\neq z$, let $A^{\prime}=A\backslash \{\eta\}\cup \{v^{\prime}\}$ with $\eta\in A$,
then there is a fan $\mathcal{O}^{\prime}$ which contains $r+2s$ internally disjoint paths
from $z$ to $A^{\prime}$ in $D[1]$.
Let $P_{1}^{\prime},\ldots,P_{r-1}^{\prime},\tilde{P}_{1}^{\prime},\ldots,\tilde{P}_{s}^{\prime}$
be respectively obtained from
$P_{1},\ldots,P_{r},\tilde{P}_{1},\ldots,\tilde{P}_{s}$ shown in formulas $(4)$ and $(5)$ by
replacing the paths in the fan $\mathcal{O}$ with the
corresponding paths in the fan $\mathcal{O}^{\prime}$.
If $z\in A$, without loss of generality, assume that $z=z_{1r}$,
then regard the path $\mathcal{O}[z,z_{1r}]$ as the vertex $z$ in the path $P_{r}$.
Let $\eta=z_{1r}$ and
$\tilde{P}_{s+1}=\mathcal{O}^{\prime}[z,v^{\prime}]v^{\prime}vxx^{\prime}
\hat{U}_{0}[x^{\prime},y^{\prime}]y^{\prime}y$.
Then $P_{1}^{\prime},\ldots,P_{r-1}^{\prime},P_{r},\tilde{P}_{1}^{\prime},\ldots,\tilde{P}_{s}^{\prime}$
 and $\tilde{P}_{s+1}$
are our desired ID$S$-paths in $D_{k,n}$.
If $z\not\in A$,
then $z^{\prime}\in V(H)$.
Let $\eta=z_{1r}$.
We will
destroy the path $P_{r}$ and find two paths $P_{r+1}$ and $P_{r+2}$ in $D_{k,n}$.
Replace $v^{\prime}$ with $x^{\prime}$ in $B_{2}$.
Let $P_{r+1}=T_{r}[y,x]xvv^{\prime}\mathcal{O}^{\prime}[v^{\prime},z]$
and $P_{r+2}=zz^{\prime}\mathcal{B}_{1}[z^{\prime},y^{\prime}]y^{\prime}yT_{r}[y,u_{1r}]
u_{1r}u_{1r}^{\prime}\mathcal{B}_{2}[u_{1r}^{\prime},x^{\prime}]x^{\prime}x$
if $a=y^{\prime}$ and $c=z^{\prime}$.
Otherwise, $P_{r+2}=zz^{\prime}\mathcal{B}_{1}[z^{\prime},u_{1r}^{\prime}]u_{1r}^{\prime}u_{1r}
T_{r}[u_{1r},y]yy^{\prime}\mathcal{B}_{2}[y^{\prime},x^{\prime}]x^{\prime}x$.
Then $P_{1}^{\prime},\ldots,P_{r-1}^{\prime},P_{r+1},P_{r+2},
\tilde{P}_{1}^{\prime},\ldots,\tilde{P}_{s}^{\prime}$
 are our desired ID$S$-paths in $D_{k,n}$.

\indent
$\mathbf {Case~2.1.2}$. $v^{\prime}\in V(H)$.

\indent
$\mathbf {Case~2.1.2.1}$. $z\in A$.

\indent
Without loss of generality, assume that $z=z_{1r}$.
Then regard the path $\mathcal{O}[z,z_{1r}]$ as the vertex $z$ in the path $P_{r}$.
Let $A^{\prime}=A\backslash \{z_{1r}\}\cup \{z_{2s+1}\}$
with $z_{2s+1}^{\prime}\in V(H)$.
Denoted by $\mathcal{O}^{\prime}$ the fan from $z$ to $A^{\prime}$ in $D[1]$.
Let $B_{3}=\{a,b\}=\{y^{\prime},z_{2s+1}^{\prime}\}$ and $B_{4}=\{c,d\}=\{x^{\prime},v^{\prime}\}$ be the
vertex subsets of $H$.
Then there are two pairwise internally disjoint $(B_{3},B_{4})$-paths in $H$ as Lemma \ref{lem5},
denoted by $\mathcal{B}_{3}[c,a]$ and $\mathcal{B}_{4}[b,d]$.
Let $\tilde{P}_{s+1}=yy^{\prime}\mathcal{B}_{4}[y^{\prime},v^{\prime}]v^{\prime}vxx^{\prime}
\mathcal{B}_{3}[x^{\prime},z_{2s+1}^{\prime}] z_{2s+1}^{\prime}z_{2s+1}\mathcal{O}^{\prime}[z_{2s+1},$
$z]$ if $a=z_{2s+1}^{\prime}$ and $c=x^{\prime}$
(see Figure~\ref{Fig.6}$~(a)$).
Otherwise, $\tilde{P}_{s+1}=yy^{\prime}\mathcal{B}_{4}[y^{\prime},x^{\prime}]x^{\prime}xvv^{\prime}
\mathcal{B}_{3}[v^{\prime},z_{2s+1}^{\prime}]z_{2s+1}^{\prime}z_{2s+1}$
$\mathcal{O}^{\prime}[z_{2s+1},z]$.
Then $P_{1}^{\prime},\ldots,P_{r-1}^{\prime},$ $P_{r},\tilde{P}_{1}^{\prime},\ldots,\tilde{P}_{s}^{\prime}$
and $\tilde{P}_{s+1}$ are our desired ID$S$-paths in $D_{k,n}$,
where $P_{1}^{\prime},\ldots,P_{r-1}^{\prime},\tilde{P}_{1}^{\prime},\ldots,\tilde{P}_{s}^{\prime}$
are respectively obtained from
$P_{1},\ldots,P_{r-1},\tilde{P}_{1},\ldots,\tilde{P}_{s}$ shown in formulas $(4)$ and $(5)$
by replacing the paths in the fan $\mathcal{O}$ with the
 corresponding paths in the fan $\mathcal{O}^{\prime}$.
\begin{figure}[ht]
\begin{center}
\scalebox{0.5}[0.5]{\includegraphics{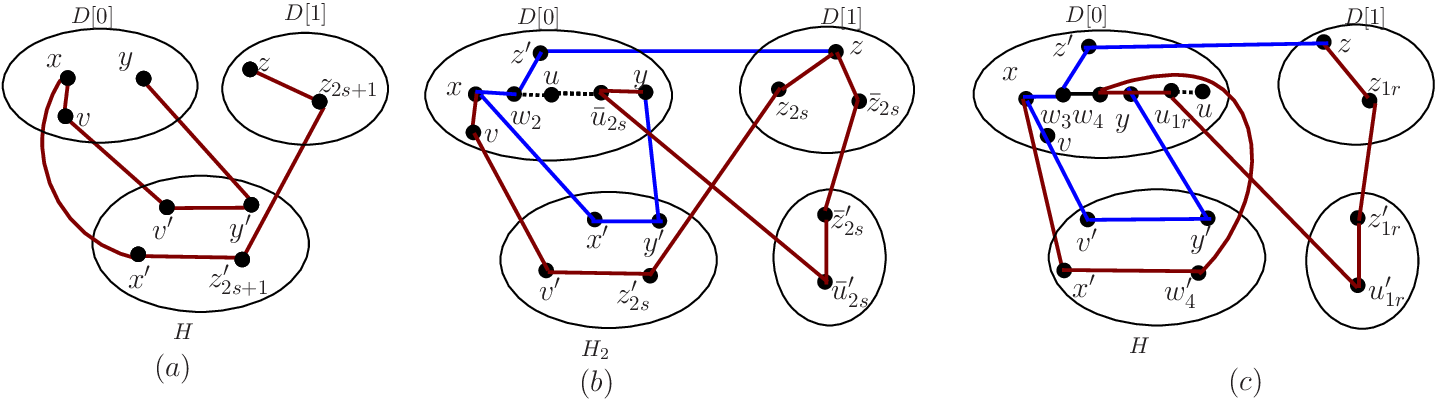}}\\
\scalebox{0.5}[0.5]{\includegraphics{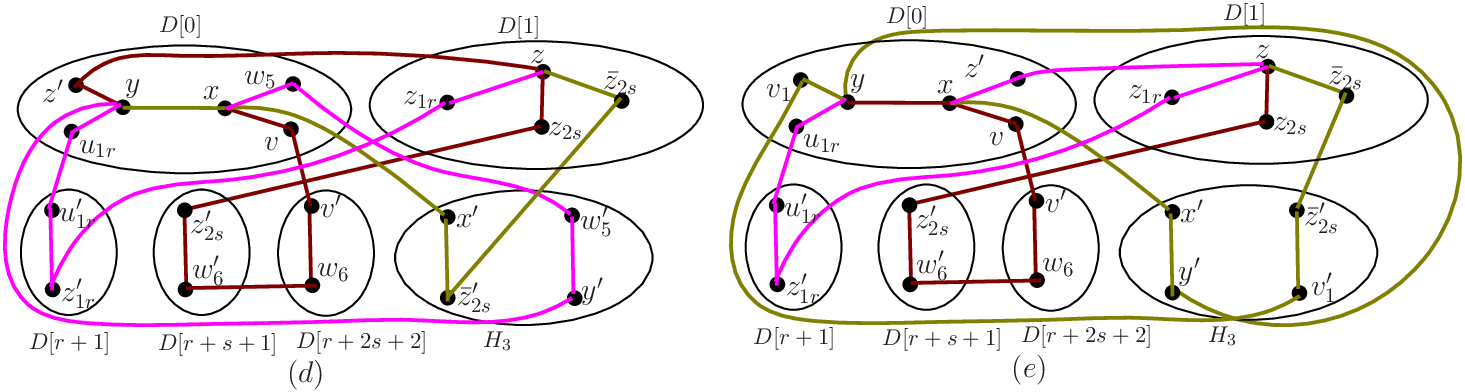}}\\
\captionsetup{font={small}}
\caption{The graphs of Case $2.1$ in Claim~3}\label{Fig.6}
\end{center}
\end{figure}

\indent
$\mathbf {Case~2.1.2.2}$. $z\not\in A$.

\indent
As $z\not\in A$,
one has that either $z^{\prime}\in V(H)$ or $z^{\prime}\in V(D[0])$ as Lemma \ref{lem1}~$(3)$.
Firstly, we consider $z^{\prime}\in V(H)$.
If $z^{\prime}\in V(H)$,
replace $z_{2s+1}^{\prime}$ with $z^{\prime}$ in $B_{3}$.
Let $\tilde{P}_{s+1}^{\prime}$ be the path obtained from
$\tilde{P}_{s+1}$ which is shown in Case 2.1.2.1 just by replacing
$\mathcal{B}_{3}[x^{\prime},z_{2s+1}^{\prime}]z_{2s+1}^{\prime}z_{2s+1}\mathcal{O}^{\prime}[z_{2s+1},z]$
(or $\mathcal{B}_{3}[v^{\prime},z_{2s+1}^{\prime}]z_{2s+1}^{\prime}z_{2s+1}\mathcal{O}^{\prime}[z_{2s+1},z]$)
with $\mathcal{B}_{3}[x^{\prime},z^{\prime}]z^{\prime}z$
(or $\mathcal{B}_{3}[v^{\prime},z^{\prime}]z^{\prime}z$).
Then $P_{1},\ldots,P_{r},\tilde{P}_{1},$ $\ldots,\tilde{P}_{s}$
and $\tilde{P}_{s+1}^{\prime}$ are our desired ID$S$-paths in $D_{k,n}$.
\\
\indent
Next, we will consider the case of $z^{\prime}\in V(D[0])$ as follows.
$\mathbf{I}:$ $z^{\prime},w\not\in V(\mathcal{T})$ for some $w\in N_{D[0]}(z^{\prime})$;
$\mathbf{II}:$ $z^{\prime}\not\in V(\mathcal{T})$ and $N_{D[0]}(z^{\prime})\subseteq V(\mathcal{T})$ and
$\mathbf{III}:$ $z^{\prime}\in V(\mathcal{T})$.
\\
\indent
$\mathbf{II}$ is divided into $\mathbf{II.1}$ and $\mathbf{II.2}$,
where $\mathbf{II.1}:$
$N_{D[0]}(z^{\prime})\cap \bigcup_{\varepsilon=1}^{r+s}V(T_{\varepsilon}[a,u])\neq\emptyset$
and
$\mathbf{II.2}:$ $N_{D[0]}(z^{\prime})\cap \bigcup_{\varepsilon=1}^{r+s}V(T_{\varepsilon}[a,u])=\emptyset$,
which implies that
$N_{D[0]}(z^{\prime})\subseteq \bigcup_{i=1}^{r}V(T_{i}[x,y])$.
Similarly, $\mathbf{III}$ is divided into $\mathbf{III.1}$ and $\mathbf{III.2}$,
where $\mathbf{III.1}:$
$z^{\prime}\in V(T_{\delta}[b,u])$ for some $\delta\in [r+s]$
and $\mathbf{III.2}:$ $z^{\prime}\in V(T_{j}[x,y])$ for some $j\in [r]$,
where $a,b\in\{x,y\}$ and $T_{\varepsilon}[a,u]$ (resp.~$T_{\delta}[b,u]$)
be the subpath of $T_{\varepsilon}$ (resp.~$T_{\delta}$) which does not contain
$\{x,y\}\backslash\{a\}$ (resp.~$\{x,y\}\backslash\{b\}$).
Furthermore, the discussions of $\mathbf{II.1}$ and $\mathbf{III.1}$ are together.
\\
\indent
$\mathbf{I}.$ If $z^{\prime},w\not\in V(\mathcal{T})$
with some $w\in N_{D[0]}(z^{\prime})$,
then $w^{\prime}\in V(H)$ as Lemma \ref{lem1}~$(3)$.
Replace $z_{2s+1}^{\prime}$ with $w^{\prime}$ in $B_{3}$.
Let $\tilde{P}_{s+1}=yy^{\prime}\mathcal{B}_{4}[y^{\prime},v^{\prime}]v^{\prime}v
xx^{\prime}\mathcal{B}_{3}[x^{\prime},w^{\prime}]w^{\prime}wz^{\prime}z$ if $a=w^{\prime}$ and $c=x^{\prime}$.
Otherwise, $\tilde{P}_{s+1}=yy^{\prime}\mathcal{B}_{4}[y^{\prime},x^{\prime}]x^{\prime}xv$
$v^{\prime}\mathcal{B}_{3}[v^{\prime},w^{\prime}]w^{\prime}wz^{\prime}z$.
Then $P_{1},\ldots,P_{r},\tilde{P}_{1},\ldots,\tilde{P}_{s}$ and $\tilde{P}_{s+1}$
are our desired ID$S$-paths in $D_{k,n}$.
\\
\indent
$\mathbf{II.1}$ (resp. $\mathbf{III.1}$).
It implies that
$N_{D[0]}[z^{\prime}]\cap V(T_{\delta}[b,u])\neq \emptyset$ for some $\delta\in [r+s]$.
Without loss of generality, assume that $\delta=r+s,b=x$
and let $w_{2}\in N_{D[0]}[z^{\prime}]\cap V(T_{r+s}[x,u])$.
Let $H_{2}=D_{k,n}[V(H)\cup V(D[r+s+1])]$. Then $\{x^{\prime},y^{\prime},z_{2s}^{\prime},v^{\prime}\}\subseteq V(H_{2})$.
Suppose that $B_{5}=\{a,b\}=\{y^{\prime},z_{2s}^{\prime}\}$ and $B_{6}=\{c,d\}=\{x^{\prime},v^{\prime}\}$.
By Lemma \ref{lem5}, there are two internally disjoint paths, say $\mathcal{B}_{5}[c,a]$ and $\mathcal{B}_{6}[b,d]$,
connecting $(B_{5},B_{6})$, in $H_{2}$.
Let $\tilde{P}_{s+1}=yy^{\prime}\mathcal{B}_{6}[y^{\prime},x^{\prime}]x^{\prime}xT_{r+s}[x,w_{2}]w_{2}z^{\prime}z$
and
$\tilde{P}_{s+2}=xvv^{\prime}\mathcal{B}_{5}[v^{\prime},z_{2s}^{\prime}]z_{2s}^{\prime}z_{2s}
\mathcal{O}[z_{2s},z]z\mathcal{O}[z,\bar{z}_{2s}]\bar{z}_{2s}\bar{z}_{2s}^{\prime}
\bar{R}_{2s}[$ $\bar{z}_{2s}^{\prime},\bar{u}_{2s}^{\prime}]
\bar{u}_{2s}^{\prime}\bar{u}_{2s}T_{r+s}[\bar{u}_{2s},y]$ if $b=y^{\prime}$ and $d=x^{\prime}$
(see Figure~\ref{Fig.6}$~(b)$).
Otherwise,
$\tilde{P}_{s+1}=yy^{\prime}\mathcal{B}_{6}[y^{\prime},v^{\prime}]v^{\prime}vx T_{r+s}[x,w_{2}]w_{2}z^{\prime}z$
and
$\tilde{P}_{s+2}=xx^{\prime}\mathcal{B}_{5}[x^{\prime},z_{2s}^{\prime}]z_{2s}^{\prime}z_{2s}
\mathcal{O}[z_{2s},z]z\mathcal{O}[z,\bar{z}_{2s}]\bar{z}_{2s}\bar{z}_{2s}^{\prime}
\bar{R}_{2s}[$ $\bar{z}_{2s}^{\prime},\bar{u}_{2s}^{\prime}]
\bar{u}_{2s}^{\prime}\bar{u}_{2s}T_{r+s}[\bar{u}_{2s},y]$.
Then $P_{1},\ldots,P_{r},\tilde{P}_{1},\ldots,\tilde{P}_{s-1},\tilde{P}_{s+1}$ and $\tilde{P}_{s+2}$
are $\lfloor\frac{2n+3k}{4}\rfloor$ ID$S$-paths in $D_{k,n}$.
\\
\indent
$\mathbf{II.2}.$
If $z^{\prime}\not\in V(\mathcal{T})$ and $N_{D[0]}(z^{\prime})\subseteq\bigcup_{i=1}^{r}V(T_{i}[x,y])$,
then there exists at least one path in $\mathcal{T}$, say $T_{r}$, such that
$|N_{D[0]}(z^{\prime})\cap V(T_{r}[x,y])|\geq 2$
according to $d_{D[0]}(z^{\prime})=n+k-2$ and $r\leq \lfloor\frac{2n+3(k-1)}{4}\rfloor$.
Without loss of generality, assume that $\{w_{3},w_{4}\}\subseteq N_{D[0]}(z^{\prime})\cap V(T_{r}[x,y])$.
Suppose that $d_{T_{r}}(x,w_{3})<d_{T_{r}}(x,w_{4})$.
Recall that $B_{3}=\{a,b\}=\{y^{\prime},z_{2s+1}^{\prime}\}$ and $B_{4}=\{c,d\}=\{x^{\prime},v^{\prime}\}$.
Replace $z_{2s+1}^{\prime}$ of $B_{3}$ with $w_{4}^{\prime}$.
Let $\tilde{P}_{s+1}=yy^{\prime}\mathcal{B}_{4}[y^{\prime},v^{\prime}]v^{\prime}vx T_{r}[x,w_{3}]w_{3}z^{\prime}z$
and $\tilde{P}_{s+2}=xx^{\prime}\mathcal{B}_{3}[x^{\prime},w_{4}^{\prime}]w_{4}^{\prime}w_{4}yT_{r}[y,u_{1r}]
u_{1r}u_{1r}^{\prime}R_{1r}[u_{1r}^{\prime},z_{1r}^{\prime}]z_{1r}^{\prime}$$z_{1r}\mathcal{O}[z_{1r},z]$ if $b=y^{\prime}$
and $d=v^{\prime}$
 (see Figure~\ref{Fig.6}$~(c)$).
Otherwise,
$\tilde{P}_{s+1}=yy^{\prime}\mathcal{B}_{4}[y^{\prime},x^{\prime}]x^{\prime}xT_{r}[x,w_{3}]w_{3}z^{\prime}z$
and $\tilde{P}_{s+2}=xvv^{\prime}\mathcal{B}_{3}[v^{\prime},w_{4}^{\prime}]w_{4}^{\prime}w_{4}yT_{r}[y,$ $u_{1r}]
u_{1r}u_{1r}^{\prime}R_{1r}[u_{1r}^{\prime},z_{1r}^{\prime}]z_{1r}^{\prime}$$z_{1r}\mathcal{O}[z_{1r},z]$.
 Then $P_{1},\ldots,P_{r-1},\tilde{P}_{1},\ldots,\tilde{P}_{s},\tilde{P}_{s+1}$ and $\tilde{P}_{s+2}$
are our desired ID$S$-paths in $D_{k,n}$.
\\
\indent
$\mathbf{III.2}:$ If $z^{\prime}\in V(T_{j}[x,y])$ for some $j\in [r]$,
without loss of generality, assume that $j=r$.
\\
\indent
If $yz^{\prime}\not\in E(T_{r}[x,y])$,
then there is a vertex $\bar{w}_{4}$ such that $y\bar{w}_{4}\in E(T_{r}[z^{\prime},y])$ and
$\bar{w}_{4}^{\prime}\in V(H)$.
Replacing $z_{2s+1}^{\prime}$ of $B_{3}$ with $\bar{w}_{4}^{\prime}$.
Let $\tilde{P}_{s+1}^{\prime}$ be the path obtained from $\tilde{P}_{s+1}$ in II.2 by
replacing $T_{r}[x,w_{3}]w_{3}$ with $T_{r}[x,z^{\prime}]$,
$\tilde{P}_{s+2}^{\prime}$ be the path obtained from $\tilde{P}_{s+2}$ in II.2 by
replacing $\mathcal{B}_{3}[x^{\prime},w_{4}^{\prime}]w_{4}^{\prime}w_{4}$ with
$\mathcal{B}_{3}[x^{\prime},\bar{w}_{4}^{\prime}]\bar{w}_{4}^{\prime}\bar{w}_{4}$.
Then $P_{1},\ldots,P_{r-1},\tilde{P}_{1},\ldots,\tilde{P}_{s},\tilde{P}_{s+1}^{\prime}$
and $\tilde{P}_{s+2}^{\prime}$
are our desired ID$S$-paths in $D_{k,n}$.
\\
\indent
If $yz^{\prime}\in E(T_{r}[x,y])$ and $xz^{\prime}\not\in E(T_{r}[x,y])$,
without loss of generality,
let $xw_{5}\in E(T_{r}[x,z^{\prime}])$.
By Lemma \ref{lem1},
suppose that $\{v^{\prime},w_{6}\}\subseteq V(D[r+2s+2])$ such that $w_{6}^{\prime}\in V(D[r+s+1])$
and $\{y^{\prime},x^{\prime},w_{5}^{\prime}\}\subseteq V(D_{k,n}[\bigcup_{\ell=r+2s+3}^{t_{k-1,n}}V(D[\ell])])$.
Let $H_{3}=D_{k,n}[\bigcup_{\ell=r+2s+3}^{t_{k-1,n}}V(D[\ell])\cup V(D[r+2s+1])]$,
$B_{7}=\{a,b\}=\{y^{\prime},\bar{z}_{2s}^{\prime}\}$ and $B_{8}=\{c,d\}=\{x^{\prime},w_{5}^{\prime}\}$.
Then there exist two internally disjoint $(B_{7},B_{8})$-paths,
say $\mathcal{B}_{7}[c,a]$ and $\mathcal{B}_{8}[b,d]$, in $H_{3}$.
Moreover, assume that $\hat{U}_{1}$ and $\hat{U}_{2}$ are the $(v^{\prime},w_{6})$-path
and $(w_{6}^{\prime},z_{2s}^{\prime})$-path in $D[r+2s+2]$ and $D[r+s+1]$, respectively.
We will destroy the paths $P_{r}$ and $\tilde{P}_{s}$ and find three other paths.
Let $\tilde{P}_{s+1}=xvv^{\prime}\hat{U}_{1}[v^{\prime},w_{6}]w_{6}w_{6}^{\prime}
\hat{U}_{2}[w_{6}^{\prime},z_{2s}^{\prime}]z_{2s}^{\prime}z_{2s}$$\mathcal{O}[z_{2s},z]zz^{\prime}y$,
$\tilde{P}_{s+2}=\mathcal{O}[z,z_{1r}]z_{1r}z_{1r}^{\prime}
R_{1r}[z_{1r}^{\prime},u_{1r}^{\prime}]u_{1r}^{\prime}u_{1r}T_{r}[$ $u_{1r},y]y$$y^{\prime}\mathcal{B}_{8}[y^{\prime},w_{5}^{\prime}]w_{5}^{\prime}w_{5}x$
and $\tilde{P}_{s+3}=T_{r+s}[y,x]xx^{\prime}\mathcal{B}_{7}[x^{\prime},\bar{z}_{2s}^{\prime}]\bar{z}_{2s}^{\prime}\bar{z}_{2s}
\mathcal{O}[\bar{z}_{2s},z]$ if $b=y^{\prime}$ and $d=w_{5}^{\prime}$
(see Figure~\ref{Fig.6}$~(d)$).
Otherwise,
$\tilde{P}_{s+2}=\mathcal{O}[z,z_{1r}]z_{1r}z_{1r}^{\prime}
R_{1r}[z_{1r}^{\prime},u_{1r}^{\prime}]u_{1r}^{\prime}u_{1r}T_{r}[u_{1r},y]y$
$y^{\prime}\mathcal{B}_{8}[y^{\prime},x^{\prime}]x^{\prime}x$
and $\tilde{P}_{s+3}=T_{r+s}[y,x]xw_{5}w_{5}^{\prime}\mathcal{B}_{7}[w_{5}^{\prime},$
$\bar{z}_{2s}^{\prime}]\bar{z}_{2s}^{\prime}\bar{z}_{2s}
\mathcal{O}[\bar{z}_{2s},z]$.
Then $P_{1},\ldots,P_{r-1},\tilde{P}_{1},\ldots,$
$\tilde{P}_{s-1},\tilde{P}_{s+1},\tilde{P}_{s+2}$ and $\tilde{P}_{s+3}$
are $\lfloor\frac{2n+3k}{4}\rfloor$ ID$S$-paths in $D_{k,n}$.
\\
\indent
If $\{xz^{\prime},yz^{\prime}\}\subseteq E(T_{r}[x,y])$,
recall that $v,v_{1}\in (N_{D[0]}(x)\cup N_{D[0]}(y)\cup N_{D[0]}(u))\backslash V(\mathcal{T})$.
As $d_{\mathcal{T}}(u)=d_{D[0]}(u)$,
we also have that
$v_{1}\in  (N_{D[0]}(x)\cup N_{D[0]}(y))\backslash V(\mathcal{T})$.
If $v_{1}\in N_{D[0]}(x)$,
then replace $w_{5}^{\prime}$ with $v_{1}^{\prime}$ in $B_{8}$.
Let $\tilde{P}_{s+1}^{\prime}=\tilde{P}_{s+1},\tilde{P}_{s+3}^{\prime}=\tilde{P}_{s+3}$
(resp.~$\tilde{P}_{s+1}^{\prime}=\tilde{P}_{s+1},\tilde{P}_{s+2}^{\prime}=\tilde{P}_{s+2}$)
and $\tilde{P}_{s+2}^{\prime}$ (resp.~$\tilde{P}_{s+3}^{\prime}$) be obtained
from $\tilde{P}_{s+2}$ (resp.~$\tilde{P}_{s+3}$)
by replacing $\mathcal{B}_{8}[y^{\prime},w_{5}^{\prime}]w_{5}w_{5}^{\prime}$
(resp.~$w_{5}w_{5}^{\prime}\mathcal{B}_{7}[w_{5}^{\prime},\bar{z}_{2s}^{\prime}]$)
with $\mathcal{B}_{8}[y^{\prime},v_{1}^{\prime}]v_{1}v_{1}^{\prime}$
(resp.~$v_{1}v_{1}^{\prime}\mathcal{B}_{7}[v_{1}^{\prime},\bar{z}_{2s}^{\prime}]$),
where $\tilde{P}_{s+1},\tilde{P}_{s+2}$ and $\tilde{P}_{s+3}$ are shown in
the case of
$yz^{\prime}\in E(T_{r}[x,y])$ and $xz^{\prime}\not\in E(T_{r}[x,y])$.
Then $P_{1},\ldots,P_{r-1},\tilde{P}_{1},\ldots,$
$\tilde{P}_{s-1},\tilde{P}_{s+1}^{\prime},\tilde{P}_{s+2}^{\prime}$ and $\tilde{P}_{s+3}^{\prime}$
are $\lfloor\frac{2n+3k}{4}\rfloor$ ID$S$-paths in $D_{k,n}$.
If $v_{1}\in N_{D[0]}(y)$,
let each element in $\{v^{\prime},w_{6},w_{6}^{\prime},x^{\prime},$ $y^{\prime},\bar{z}_{2s}^{\prime},
H_{3},B_{7},$
$B_{8},\mathcal{B}_{7}[c,a],\mathcal{B}_{8}[b,$
$d],\hat{U}_{1}[v^{\prime},w_{6}],
\hat{U}_{2}[w_{6}^{\prime},z_{2s}^{\prime}]\}$
be the same and has the same properties as that in the case of
$yz^{\prime}\in E(T_{r}[x,y])$ and $xz^{\prime}\not\in E(T_{r}[x,y])$.
Replace $\bar{z}_{2s}^{\prime}$ and $w_{5}^{\prime}$ in $B_{7}$ and $B_{8}$ with $v_{1}^{\prime}$ and $\bar{z}_{2s}^{\prime}$, respectively.
We will destroy the paths $P_{r}$ and $\tilde{P}_{s}$ and find three other paths.
Let $\tilde{P}_{s+1}=T_{r+s}[y,x]xvv^{\prime}\hat{U}_{1}[v^{\prime},w_{6}]w_{6}w_{6}^{\prime}$
$\hat{U}_{2}[w_{6}^{\prime},z_{2s}^{\prime}]z_{2s}^{\prime}z_{2s}\mathcal{O}[z_{2s},z]$,
$\tilde{P}_{s+2}=xz^{\prime}z\mathcal{O}[z,z_{1r}]z_{1r}z_{1r}^{\prime}R_{1r}[z_{1r}^{\prime},u_{1r}^{\prime}]
u_{1r}^{\prime}u_{1r}T_{r}[u_{1r},y]$
and
$\tilde{P}_{s+3}=xx^{\prime}\mathcal{B}_{7}[x^{\prime},y^{\prime}]y^{\prime}yv_{1}v_{1}^{\prime}
\mathcal{B}_{8}[v_{1}^{\prime},
\bar{z}_{2s}^{\prime}]
\bar{z}_{2s}^{\prime}\bar{z}_{2s}$ $\mathcal{O}[\bar{z}_{2s},z]$ if $a=y^{\prime}$ and $c=x^{\prime}$
(see Figure~\ref{Fig.6}$~(e)$).
Otherwise, $\tilde{P}_{s+3}=xx^{\prime}\mathcal{B}_{7}[x^{\prime},v_{1}^{\prime}]v_{1}^{\prime}v_{1}yy^{\prime}
\mathcal{B}_{8}[y^{\prime},\bar{z}_{2s}^{\prime}]$
$\bar{z}_{2s}^{\prime}\bar{z}_{2s}\mathcal{O}[\bar{z}_{2s},z]$.
Then $P_{1},\ldots,P_{r-1},\tilde{P}_{1},\ldots,\tilde{P}_{s-1},\tilde{P}_{s+1},$
$\tilde{P}_{s+2}$ and $\tilde{P}_{s+3}$
are $\lfloor\frac{2n+3k}{4}\rfloor$ ID$S$-paths in $D_{k,n}$.

\indent
$\mathbf {Case~2.2}$. $d_{\mathcal{T}}(u)+1=d_{D[0]}(u)$.

\indent
Let each one in $\{H_{1},H_{2},B_{i},B_{j},\mathcal{B}_{i}[c,a],\mathcal{B}_{j}[b,d]|i=1,3,5;j=2,4,6.\}$
be the same as that in Case~$2.1$.
Recall that $|A|=n+k-3$. Then add a new vertex $z_{2s+1}$ to $A$
and let $A^{\prime\prime}=A\cup \{z_{2s+1}\}$,
where $z_{2s+1}^{\prime}\in V(H)$.
By Fan Lemma \ref{lem4}, there is a fan $\mathcal{O}^{\prime\prime}$ which contains
$r+2s+1$ internally disjoint paths from $z$ to $A^{\prime\prime}$ in $D[1]$.
Let $P_{1}^{\prime\prime},\ldots,P_{r}^{\prime\prime},\tilde{P}_{1}^{\prime\prime},\ldots,
\tilde{P}_{s}^{\prime\prime}$
be respectively obtained from
$P_{1},\ldots,P_{r},\tilde{P}_{1},\ldots,\tilde{P}_{s}$ shown in formulas $(4)$ and $(5)$ by
replacing the paths in the fan $\mathcal{O}$ with the
corresponding paths in the fan $\mathcal{O}^{\prime\prime}$.
Recall that there are at lest two vertices $v$ and $v_{1}$
such that $v,v_{1}\in (N_{D[0]}(x)\cup N_{D[0]}(y)\cup N_{D[0]}(u))\backslash V(\mathcal{T})$.
Since $d_{\mathcal{T}}(u)+1=d_{D[0]}(u)$,
there exists at least one vertex, say $v$, such that $v\in (N_{D[0]}(x)\cup N_{D[0]}(y))\backslash V(\mathcal{T})$.
Without loss of generality, let $v\in N_{D[0]}(x)$. Then $v^{\prime}\in V(D[1]\cup H)$.
\\
\indent
If $v^{\prime}\in V(H)$,
let $\tilde{P}_{s+1}^{\prime}$ be the path obtained from
$\tilde{P}_{s+1}$ in Case~$2.1.2.1$ by
replacing $\mathcal{O}^{\prime}[z_{2s+1},z]$ with $\mathcal{O}^{\prime\prime}[z_{2s+1},z]$ (see Figure~\ref{Fig.6}$~(a)$).
Then $P_{1}^{\prime\prime},\ldots,P_{r}^{\prime\prime},\tilde{P}_{1}^{\prime\prime},\ldots,
\tilde{P}_{s}^{\prime\prime}$ and $\tilde{P}_{s+1}^{\prime}$
are our desired ID$S$-paths in $D_{k,n}$.
\\
\indent
If $v^{\prime}\in V(D[1])$ and $z\not\in A\cup \{v^{\prime}\}$, then $z^{\prime}\in V(H)$.
Replace $z_{2s}^{\prime}$ in $B_{5}$ with $x^{\prime}$,
$v^{\prime}$ and $x^{\prime}$ with $z^{\prime}$ and $z_{2s+1}^{\prime}$, respectively, in $B_{6}$.
Let $\tilde{P}_{s+1}=yy^{\prime}B_{6}[y^{\prime},z^{\prime}]z^{\prime}z
\mathcal{O}^{\prime\prime}[z,z_{2s+1}]
z_{2s+1}z_{2s+1}^{\prime}B_{5}[z_{2s+1}^{\prime},x^{\prime}]$ $x^{\prime}x$ if $a=x^{\prime}$ and $c=z_{2s+1}^{\prime}$.
Otherwise, $\tilde{P}_{s+1}=yy^{\prime}B_{6}[y^{\prime},z_{2s+1}^{\prime}]z_{2s+1}^{\prime}z_{2s+1}
\mathcal{O}^{\prime\prime}[z_{2s+1},z]zz^{\prime}B_{5}[z^{\prime},$ $x^{\prime}]x^{\prime}x$.
Then $P_{1}^{\prime\prime},\ldots,P_{r}^{\prime\prime},\tilde{P}_{1}^{\prime\prime},\ldots,
\tilde{P}_{s}^{\prime\prime}$ and $\tilde{P}_{s+1}$
are our desired ID$S$-paths in $D_{k,n}$.
\\
\indent
If $v^{\prime}\in V(D[1])$ and $z\in A\cup \{v^{\prime}\}$,
then either $z=v^{\prime}$ or $z\in A$.
If $z=v^{\prime}$,
then $P_{1}^{\prime\prime},\ldots,P_{r}^{\prime\prime},\tilde{P}_{1}^{\prime\prime},\ldots,
\tilde{P}_{s}^{\prime\prime}$ and $\tilde{P}_{s+1}$
are our desired ID$S$-paths in $D_{k,n}$,
where $\tilde{P}_{s+1}$ is shown in the case of $v^{\prime}=z$ in Case~$2.1.1$.
If $z\in A$, without loss of generality, assume that $z=z_{1r}$,
then regard the path $\mathcal{O}[z,z_{1r}]$ as the vertex $z$ in the path $P_{r}$.
Let $A^{\prime\prime\prime}=A^{\prime\prime}\backslash\{z_{1r}\}\cup \{v^{\prime}\}$.
By Fan Lemma \ref{lem4}, there is a fan $\mathcal{O}^{\prime\prime\prime}$ from $z$ to $A^{\prime\prime\prime}$ in $D[1]$.
Let $P_{1}^{\prime\prime\prime},\ldots,P_{r-1}^{\prime\prime\prime},\tilde{P}_{1}^{\prime\prime\prime},\ldots,
\tilde{P}_{s}^{\prime\prime\prime}$
be respectively obtained from
$P_{1},\ldots,P_{r-1},\tilde{P}_{1},\ldots,\tilde{P}_{s}$ shown in formulas $(4)$ and $(5)$ by
replacing the paths in the fan $\mathcal{O}$ with the
corresponding paths in the fan $\mathcal{O}^{\prime\prime\prime}$.
Let $\tilde{P}_{s+1}^{\prime}$ be obtained from $\tilde{P}_{s+1}$ which shown in the case of
$z\in A$ of Case~$2.1.1$ by replacing $\mathcal{O}^{\prime}[z,v^{\prime}]$
with $\mathcal{O}^{\prime\prime\prime}[z,v^{\prime}]$.
Then $P_{1}^{\prime\prime\prime},\ldots,P_{r-1}^{\prime\prime\prime},P_{r},
\tilde{P}_{1}^{\prime\prime\prime},\ldots,
\tilde{P}_{s}^{\prime\prime\prime}$ and $\tilde{P}_{s+1}^{\prime}$
are our desired ID$S$-paths in $D_{k,n}$.

\indent
$\mathbf {Case~2.3}$. $d_{\mathcal{T}}(u)+2\leq d_{D[0]}(u)$.

\indent
Let each one in $\{B_{3},B_{4},\mathcal{B}_{3}[c,a],\mathcal{B}_{4}[b,d]\}$ be the same as that in Case~$2.1$.
Recall that $|A|\leq n+k-4$, then add two new vertices $z_{2s+1}$ and $z_{2s+2}$ to $A$
and let $A^{\prime}=A\cup \{z_{2s+1},z_{2s+2}\}$,
where $\{z_{2s+1}^{\prime},z_{2s+2}^{\prime}\}\subseteq V(H)$.
By Fan Lemma \ref{lem4},
there is a fan $\mathcal{O}^{\prime}$ which
contains $r+2s+2$ internally disjoint paths from $z$ to $A^{\prime}$ in $D[1]$.
Replace $z_{2s+1}^{\prime}$ of $B_{3}$ with $x^{\prime}$ and
$x^{\prime}$ (resp.~$v^{\prime}$) of $B_{4}$ with $z_{2s+1}^{\prime}$ (resp.~$z_{2s+2}^{\prime}$) in $H$.
Let $\tilde{P}_{s+1}=yy^{\prime}\mathcal{B}_{4}[y^{\prime},z_{2s+p}^{\prime}]z_{2s+p}^{\prime}z_{2s+p}
\mathcal{O}^{\prime}[z_{2s+p},z]z\mathcal{O}^{\prime}[z,z_{2s+q}]$ $z_{2s+q}z_{2s+q}^{\prime}
\mathcal{B}_{3}[z_{2s+q}^{\prime}$
$,x^{\prime}]x^{\prime}x$
with $\{p,q\}=\{1,2\}$.
Then $P_{1}^{\prime},\ldots,P_{r}^{\prime},\tilde{P}_{1}^{\prime},\ldots,\tilde{P}_{s}^{\prime}$ and $\tilde{P}_{s+1}$
are our desired ID$S$-paths in $D_{k,n}$,
where $P_{1}^{\prime},\ldots,P_{r}^{\prime},\tilde{P}_{1}^{\prime},\ldots,\tilde{P}_{s}^{\prime}$ are
respectively obtained from
$P_{1},\ldots,P_{r},\tilde{P}_{1},\ldots,\tilde{P}_{s}$ shown in formulas $(4)$ and $(5)$
by replacing the paths in the fan $\mathcal{O}$ with the
 corresponding paths in the fan $\mathcal{O}^{\prime}$.
$\hfill\blacksquare$
\\
\indent
By Claims~$1$-~$3$, one has that $\pi_{3}(D_{k,n})\geq\lfloor\frac{2n+3k}{4}\rfloor$
for $k\geq 0$ and $n\geq 6$.
Therefore, $\pi_{3}(D_{k,n})=\lfloor\frac{2n+3k}{4}\rfloor$ for $k\geq 0$ and $n\geq 6$.
\end{proof}

\section{Conclusion}
The $r$-path connectivity is a generalization of the traditional connectivity.
In this paper, we prove that at least $\lfloor\frac{2n+3k}{4}\rfloor$ internally disjoint paths
connecting any three distinct vertices can be constructed in the
$k$-dimensional data center network with $n$-port switches $D_{k,n}$.
Furthermore, the result $\pi_{3}(D_{k,n})=\lfloor\frac{2n+3k}{4}\rfloor$ is proved.
Determine the
$r$-path connectivity with $r\geq 4$ of the data center network would be interesting in the future.

\section*{Acknowledgements}
This work was supported by
the National Natural Science Foundation of China (Nos. 11971054, 11731002 and 12161141005)
and the Fundamental Research Funds for the Central Universities (2022JBCG003).

\section*{Appendix}{The proof of Lemma \ref{thm21}}

We first introduce some definitions and notations that will be used in the proof of Lemma \ref{thm21}.
\\
\indent
For any vertex $u\in V(G)$, let $N_{e}(u)$ denote the set of edges which are incident to $u$ in $G$.
For any path $P$ in $G$ and any two vertices $u,v\in V(P)$, let $l(P)$ be the length of $P$
and $P[u,v]$ be the subpath of $P$ connecting $u$ and $v$.
Let $C=\{e_{1},e_{2},\ldots,e_{t}\}$ be the subset of $E(G)$,
then $V(C)=\bigcup_{i=1}^{t}V(e_{i})$.

\begin{proof}
Let $x,y$ and $z$ be any three distinct vertices of $G$ with $|N_{G}(x)\cap N_{G}(y)\cap N_{G}(z)|=r$,
$\mathcal{T}$ be a set of $\pi_{3}(G)$ internally disjoint $\{x,y,z\}$-paths in $G$ and $A=N_{G}(x)\cup N_{G}(y)\cup N_{G}(z)$
such that
\\
\indent
(1) Only one vertex of $x,y$ and $z$ has degree two in $P$ for any path $P\in \mathcal{T}$;
\\
\indent
(2) Subject to (1), $|E(\mathcal{T})\cap E(G[\{x,y,z\}])|$ is as large as possible;
\\
\indent
(3) Subject to (1) and (2), $|A\cap V(\mathcal{T})|$ is as small as possible.

$\mathbf {Claim~1}$. $E(G[\{x,y,z\}])\subseteq E(\mathcal{T})$.

\noindent {\it Proof of Claim 1.}
To contrary, suppose that there exists an edge
$e\in E(G[\{x,y,z\}])$ such that $e\not\in E(\mathcal{T})$.
Notice that $|E(G[\{x,y,z\}])|\leq 2$ and $|\mathcal{T}|\geq \pi_{3}(G)\geq 2$.
Then there exists a path $T\in \mathcal{T}$ with $l(T)\geq 3$.
Without loss of generality, assume that $e=xy$,
then $e\not\in E(T)$ as $e\not \in E(\mathcal{T})$.
By symmetry, we only consider $d_{T}(z)=2$ or $d_{T}(x)=2$.
If $d_{T}(z)=2$,
without loss of generality, assume that $l(T[x,z])\geq 2$ as $l(T)\geq 3$.
Let $T_{1}=xyT[y,z]z$ and $\mathcal{T}_{1}=(\mathcal{T}\backslash \{T\})\cup \{T_{1}\}$.
Then $|E(\mathcal{T}_{1})\cap E(G[\{x,y,z\}])|>|E(\mathcal{T})\cap E(G[\{x,y,z\}])|$ subject to (1),
which contradicts to (2).
If $d_{T}(x)=2$,
let $T_{2}=yxT[x,z]z$ and $\mathcal{T}_{2}=(\mathcal{T}\backslash \{T\})\cup \{T_{2}\}$.
Then $|E(\mathcal{T}_{2})\cap E(G[\{x,y,z\}])|>|E(\mathcal{T})\cap E(G[\{x,y,z\}])|$,
which contradicts to (2).
Thus $E(G[\{x,y,z\}])\subseteq E(\mathcal{T})$.$\hfill\blacksquare$
\\
\indent
If $3k-4\pi_{3}(G)\geq r+1$
(that is  $d_{G}(x)+d_{G}(y)+d_{G}(z)-4|\mathcal{T}|=3k-4\pi_{3}(G)\geq r+1$),
by Claim 1,
then there are $t$ edges, say $e_{1},\ldots, e_{t}$, such that each of them is incident to
only one vertex in $\{x,y,z\}$ and
$\{e_{1},\ldots, e_{t}\}=(N_{e}(x)\cup N_{e}(y)\cup N_{e}(z))\backslash E(\mathcal{T})$,
where $t=3k-4\pi_{3}(G)\geq r+1$.
Let $(\bigcup _{j=1}^{t} V(e_{j}))\backslash \{x,y,z\}=B$,
it implies that $1\leq |B|\leq t$.

$\mathbf {Claim~2}$. $B\not\subseteq V(\mathcal{T})$.

\noindent {\it Proof of Claim 2.}
To contrary, suppose that $B\subseteq V(\mathcal{T})$.
For any vertex $\mu\in B$, then $\mu\in V(T)$ with some $T\in \mathcal{T}$.
Without loss of generality, assume that $d_{T}(y)=2$.
Let $N_{T}(x)=\{x_{11}\},N_{T}(y)=\{y_{11},y_{12}\}$,
$N_{T}(z)=\{z_{11}\}$ and $d_{T}(x,y_{11})< d_{T}(x,y_{12})$.

$\mathbf {Case~1}$. $B\not\subseteq N_{G}(x)\cap N_{G}(y)\cap N_{G}(z)$.

Since $B\not\subseteq N_{G}(x)\cap N_{G}(y)\cap N_{G}(z)$,
one has that there exists a vertex, say $\mu_{1}$, in $B$, such that
$\mu_{1}\not\in N_{G}(x)\cap N_{G}(y)\cap N_{G}(z)$.
Moreover, $\mu_{1}\in A$ as $\mu_{1}\in B$ and $B\subseteq A$.
As $\mu_{1}\in A\cap B$,
then $\alpha \mu_{1}\in \{e_{1},\ldots,e_{t}\}$ with some $\alpha \in \{x,y,z\}$.
Recall that $\mu_{1}\in V(T)$,
we consider the following cases.

$\mathbf {Case~1.1}$. $\alpha=y$.

As $yy_{11},yy_{12}\in E(T)$,
one has that $\mu_{1} \not\in \{y_{11},y_{12}\}$.
Since $\mu_{1}\in V(T)$,
without loss of generality, assume that $\mu_{1} \in V(T[x,y])$.
Let $P=xT[x,\mu_{1}]\mu_{1} yT[y,z]z$ (see Figure \ref{Fig.a}~(1)),
and $\mathcal{P}=(\mathcal{T}\backslash \{T\})\cup \{P\}$.
Then $V(\mathcal{P})\subseteq V(\mathcal{T})$,
$y_{11}\in N_{G}(y)\backslash V(\mathcal{P})$
and $\mathcal{P}$ satisfies (1) and (2).
Notice that $y_{11}\in V(\mathcal{T})$.
Then $|A\cap V(\mathcal{P})|< |A\cap V(\mathcal{T})|$,
which contradicts to (3).

$\mathbf {Case~1.2}$. $\alpha=x$ or $\alpha=z$.

Without loss of generality, assume that $\alpha=x$.

$\mathbf {Case~1.2.1}$. $\mu_{1}\neq y_{12}$.

If $\mu_{1} \in V(T[x,y])$,
as $xx_{11}\in E(T)$, then $\mu_{1}\neq x_{11}$.
Let $P_{1}=x\mu_{1} T[\mu_{1},z]z$ (see Figure \ref{Fig.a}~(2)),
and $\mathcal{P}=(\mathcal{T}\backslash \{T\})\cup \{P_{1}\}$.
Then $V(\mathcal{P})\subseteq V(\mathcal{T})$
and $x_{11}\in N_{G}(x)\backslash V(\mathcal{P})$,
it implies that $|A\cap V(\mathcal{P})|< |A\cap V(\mathcal{T})|$ subject to (1) and (2),
which contradicts to (3).
\\
\indent
If $\mu_{1} \not\in V(T[x,y])$, as $\mu_{1}\in V(T)$,
then $\mu_{1}\in V(T[y,z])$.
Let $P_{2}=yT[y,x]x\mu_{1} T[\mu_{1},z]z$ (see Figure \ref{Fig.a}~(3)),
and $\mathcal{P}=(\mathcal{T}\backslash \{T\})\cup \{P_{2}\}$.
Then $V(\mathcal{P})\subseteq V(\mathcal{T})$.
Recall that $\mu_{1}\neq y_{12}$,
then $y_{12}\in N_{G}(y)\backslash V(\mathcal{P})$,
it implies that $|A\cap V(\mathcal{P})|< |A\cap V(\mathcal{T})|$ subject to (1) and (2),
which contradicts to (3).

$\mathbf {Case~1.2.2}$. $\mu_{1}=y_{12}$.

As $\mu_{1}\not\in N_{G}(x)\cap N_{G}(y)\cap N_{G}(z)$ and $\mu_{1}\in N_{G}(x)\cap N_{G}(y)$,
then $\mu_{1}\neq z_{11}$.
\\
\indent
Suppose that there exists a path in $\mathcal{T}$ with $z$ as the internal vertex,
without loss of generality, let $T_{1}\in \mathcal{T}$ with $d_{T_{1}}(z)=2$.
Let $P_{3}=yT[y,x]xT_{1}[x,z]z,P_{4}=x\mu_{1} yT_{1}[y,z]z$ (see Figure \ref{Fig.a}~(4)),
and $\mathcal{P}=(\mathcal{T}\backslash \{T,T_{1}\})\cup \{P_{3},P_{4}\}$.
Then $V(\mathcal{P})\subseteq V(\mathcal{T})$
and $z_{11}\in N_{G}(z)\backslash V(\mathcal{P})$,
it implies that $|A\cap V(\mathcal{P})|< |A\cap V(\mathcal{T})|$ subject to (1) and (2),
which contradicts to (3).
\\
\indent
Suppose that there is no path in $\mathcal{T}$ with $z$ as the internal vertex.
Since $|\mathcal{T}|\geq \pi_{3}(G)\geq 2$,
one has that there exists a path $T_{2}\in \mathcal{T}\backslash \{T\}$
such that $d_{T_{2}}(\beta)=2$ with some $\beta\in\{x,y\}$.
Since $d_{G}(x)=d_{G}(z)$ and $d_{\mathcal{T}}(z)\leq d_{\mathcal{T}}(x)$,
one has that $d_{G}(z)-d_{\mathcal{T}}(z)\geq d_{G}(x)-d_{\mathcal{T}}(x)$,
that is $|N_{e}(z)\backslash E(\mathcal{T})|\geq |N_{e}(x)\backslash E(\mathcal{T})|$.
As $x\mu_{1}\in N_{e}(x)\backslash E(\mathcal{T})$,
then there exists at least one vertex, say $\nu\in B$,
such that $z\nu\in N_{e}(z)\backslash E(\mathcal{T})$.
Notice that $\nu\in B$ and $B\subseteq V(\mathcal{T})$.
Then $\nu\in V(\mathcal{T})$.
\\
\indent
If $\nu\in V(T)$, then either $\nu\in V(T[x,y])$ or $\nu\in V(T[y,z])$.
If $\nu \in V(T[x,y])$, let $P_{5}=z\nu T[\nu,x]x\mu_{1} y$ (see Figure \ref{Fig.a}~(5)),
and
$\mathcal{P}=(\mathcal{T}\backslash \{T\})\cup \{P_{5}\}$.
Then $V(\mathcal{P})\subseteq V(\mathcal{T})$
and $z_{11}\in N_{G}(z)\backslash V(\mathcal{P})$,
it implies that $|A\cap V(\mathcal{P})|< |A\cap V(\mathcal{T})|$ subject to (1) and (2),
which contradicts to (3).
If $\nu\in V(T[y,z])$, we can obtain a contradiction by the similar discussions
as the proof for $\mu_{1}\in V(T[x,y])$ in Case 1.2.1.
\\
\indent
If $\nu\not\in V(T)$, as $\nu\in V(\mathcal{T})$, then $\nu \in V(T_{2})$.
If $\nu\in V(T_{2}[\beta,z])$ with some $\beta\in \{x,y\}$, we can obtain a contradiction by the similar discussions
as the proof for $\mu_{1}\in V(T[x,y])$ in Case 1.2.1.
If
$\nu \in V(T_{2}[x,y])$,
let $P_{6}=z\nu T_{2}[\nu,y]y\mu_{1}x$.
If $\beta=y$, let
$P_{7}=xT[x,y]yT_{2}[y,z]z$ (see Figure \ref{Fig.a}~(6));
If $\beta=x$,
let $P_{7}=yT[y,x]xT_{2}[x,z]z$ (see Figure \ref{Fig.a}~(7)).
Let $\mathcal{P}=(\mathcal{T}\backslash \{T,T_{2}\})\cup \{P_{6},P_{7}\}$.
Then $V(\mathcal{P})\subseteq V(\mathcal{T})$ and $z_{11}\in N_{G}(z)\backslash V(\mathcal{P})$.
It implies that $|A\cap V(\mathcal{P})|< |A\cap V(\mathcal{T})|$ subject to (1) and (2),
which contradicts to (3).

$\mathbf {Case~2}$. $B\subseteq
N_{G}(x)\cap N_{G}(y)\cap N_{G}(z)$.

As $|N_{G}(x)\cap N_{G}(y)\cap N_{G}(z)|=r$,
thus $|B|\leq r$.
That is $|B|\leq t-1=|\{e_{1},\ldots,e_{t}\}|-1$ as $t\geq r+1$.
Then there exists at least one vertex, say $\mu_{1}$, in $B$, incident to at least two edges in $\{e_{1},\ldots,e_{t}\}$.
Without loss of generality, assume that $V(e_{1})\cap V(e_{2})=\{\mu_{1}\}$.
Recall that $\mu_{1}\in V(T)$,
without loss of generality, we only consider $\beta\mu_{1}=e_{i}$ with some $\beta \in \{x,y\}$ and $i\in \{1,2\}$.
If $\beta=y$ or $\beta=x$ and $\mu_{1}\neq y_{12}$,
then the proof is similar to that of Case 1.1 or Case 1.2.1, which is omitted.
If $\beta=x$ and $\mu_{1}=y_{12}$,
recall that $V(e_{1})\cap V(e_{2})=\{\mu_{1}\}$,
then $z\mu_{1}\in \{e_{1},e_{2}\}\backslash \{e_{i}\}$ and $\mu_{1}\neq z_{11}$ as $G$ is a simple graph.
Let $P=z\mu_{1}xT[x,y]y$ (see Figure \ref{Fig.a}~(8)),
and $\mathcal{P}=(\mathcal{T}\backslash \{T\})\cup \{P\}$.
Then $V(\mathcal{P})\subseteq V(\mathcal{T})$
and $z_{11}\in N_{G}(z)\backslash V(\mathcal{P})$.
It implies that $|A\cap V(\mathcal{P})|< |A\cap V(\mathcal{T})|$ subject to (1) and (2),
which contradicts to (3).
Thus Claim 2 holds.
$\hfill\blacksquare$

\begin{figure}[ht]
\begin{center}
\scalebox{0.7}[0.7]{\includegraphics{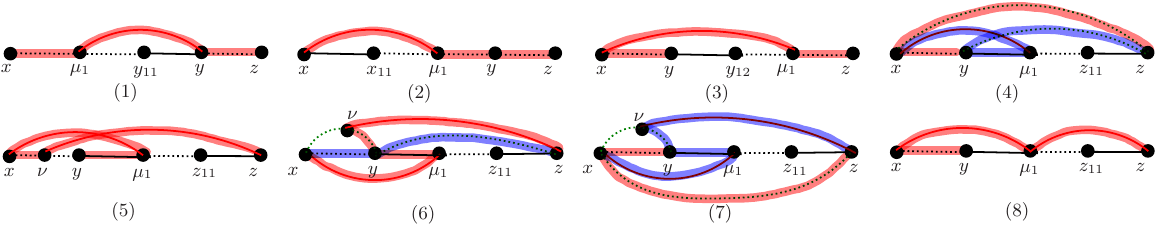}}
\caption{Some graphs of Claim 2}\label{Fig.a}
\end{center}
\end{figure}

(a) By Claim 2, $B\not\subseteq V(\mathcal{T})$ with $|B|\geq 1$,
then there exists a vertex $u\in B\backslash V(\mathcal{T})$.
Since $B\subseteq N_{G}(x)\cup N_{G}(y)\cup N_{G}(z)$,
one has that $u\in (N_{G}(x)\cup N_{G}(y)\cup N_{G}(z))\backslash V(\mathcal{T})$,
thus (a) holds.

Before prove (b), we firstly give a note here.

$\mathbf {Note}$.
Let $C=(N_{e}(x)\cup N_{e}(y)\cup N_{e}(z))\backslash E(\mathcal{T})$ and $D=V(C)\backslash \{x,y,z\}$,
or
$C=(N_{e}(x)\cup N_{e}(y)\cup N_{e}(z))\backslash (E(\mathcal{T})\cup E_{0})$ and $D=V(C)\backslash \{x,y,z\}$
with $|C|\geq 1$ and $|D|\geq 1$,
where $u\in (N_{G}(x)\cup N_{G}(y)\cup N_{G}(z))\backslash V(\mathcal{T})$ and
$E_{0}$ is the set of edges with one end in $\{u\}$
and the other end in $\{x,y,z\}$.
\\
\indent
$\mathbf {1}$. If $D\not\subseteq N_{G}(x)\cap N_{G}(y)\cap N_{G}(z)$,
the set $B$ in Case 1 of Claim 2 is replacing with $D$,
by a similar proof of that for Case 1,
then $D\not\subseteq V(\mathcal{T})$.
\\
\indent
$\mathbf {2}$. If $D\subseteq N_{G}(x)\cap N_{G}(y)\cap N_{G}(z)$
and $|D|\leq |C|-1$,
the sets $B$ and $\{e_{1},\ldots, e_{t}\}$ in Case 2 of Claim 2 are replacing with $D$ and $C$, respectively,
by a similar proof of that for Case 2,
then $D\not\subseteq V(\mathcal{T})$.

We are going to prove the result of (b).
When $\ell=3$, notice that $3k-4\pi_{3}(G)\geq r+3\geq r+1$ and by (a),
 there exist a vertex $u$ and a set $\mathcal{T}$
of $\pi_{3}(G)$ internally disjoint $\{x,y,z\}$-paths in $G$ such that
$u\in (N_{G}(x)\cup N_{G}(y)\cup N_{G}(z))\backslash V(\mathcal{T})$,
where $\mathcal{T}$ satisfies (1)-(3).
\\
\indent
Let $E_{0}=E(\{u\},\{x,y,z\})$, where $E(\{u\},\{x,y,z\})$ denote the set of edges with one end in $\{u\}$
and the other end in $\{x,y,z\}$, respectively.
Then $1\leq |E_{0}|\leq 3$.
Let $d_{G}(x)+d_{G}(y)+d_{G}(z)-4|\mathcal{T}|-|E_{0}|=g$.
As $3k-4\pi_{3}(G)\geq r+3$,
one has that $g=3k-4\pi_{3}(G)-|E_{0}|\geq r+3-|E_{0}|$,
where $g\geq r$ if $|E_{0}|=3$, and $g\geq r+1$ if $1\leq |E_{0}|\leq 2$.
By Claim 1, there are $g$ edges, say $e_{1},\ldots,e_{g}$,
each of them is incident to only one vertex of $\{x,y,z\}$ and
$\{e_{1},\ldots,e_{g}\}=(N_{e}(x)\cup N_{e}(y)\cup N_{e}(z))\backslash (E(\mathcal{T})\cup E_{0})$.
Let $C=\{e_{1},\ldots,e_{g}\}$ and $D=V(C)\backslash \{x,y,z\}$.
Then $1\leq |D|\leq g$.
As $C\cap E_{0}=\emptyset$,
then $u\not\in D$.
\\
\indent
Suppose that $D\not\subseteq N_{G}(x)\cap N_{G}(y)\cap N_{G}(z)$,
by Note 1, then $D\not\subseteq V(\mathcal{T})$.
Suppose that $D\subseteq N_{G}(x)\cap N_{G}(y)\cap N_{G}(z)$.
If $u\in N_{G}(x)\cap N_{G}(y)\cap N_{G}(z)$,
then $|E_{0}|=3$ and $|C|=g=r$.
Notice that $u\not\in D$. Then $|D|\leq r-1$.
It implies that $|D|\leq |C|-1$;
If $u\not\in N_{G}(x)\cap N_{G}(y)\cap N_{G}(z)$,
then $1\leq |E_{0}|\leq 2$ and $|C|=g\geq r+1$.
As $D\subseteq N_{G}(x)\cap N_{G}(y)\cap N_{G}(z)$,
one has that $|D|\leq r$.
Thus $|D|\leq |C|-1$.
By Note 2, one has that $D\not\subseteq V(\mathcal{T})$.
Then there exists a vertex, say $v$, in $D$, such that
$v\in D\backslash V(\mathcal{T})$.
By the definition of $D$, we have that $D\subseteq N_{G}(x)\cup N_{G}(y)\cup
N_{G}(z)$.
Then
$v\in (N_{G}(x)\cup N_{G}(y)\cup
N_{G}(z))\backslash V(\mathcal{T})$.
Recall that $u\not\in D$ and $u\in (N_{G}(x)\cup N_{G}(y)\cup N_{G}(z))\backslash V(\mathcal{T})$,
then (b) holds.
\end{proof}

{\bf Note:} Lemma \ref{thm21}~(2) can be applied to many networks such as
the $k$-ary $n$-cube $Q_{n}^{k}$.

%
%

\def\polhk#1{\setbox0=\hbox{#1}{\ooalign{\hidewidth
\lower1.5ex\hbox{`}\hidewidth\crcr\unhbox0}}}


\end{document}